%% file: qBT-R1.tex
\documentclass[11pt]{article}
\input defnAMS.tex
\usepackage{hyperref}
\usepackage{cleveref}

\usepackage{accents}
\newcommand\munderbar[1]{%
  \underaccent{\bar}{#1}}

\def\cH{{\cal H}}

\DeclareMathOperator{\app}{}
\DeclareMathOperator{\BT}{bt}
\DeclareMathOperator{\df}{d}
\DeclareMathOperator{\err}{err}

\DeclareMathOperator{\rd}{rd}

\def\scrS{\mathscr{S}}

\newtheorem{theorem}{Theorem}[section]
\newtheorem{lemma}{Lemma}[section]
\newtheorem{corollary}{Corollary}[section]

\theoremstyle{definition}

\newtheorem{remark}{Remark}[section]

\def\sss{\scriptstyle}

\numberwithin{equation}{section}
\numberwithin{figure}{section}
\numberwithin{table}{section}

\newcounter{question}

\title{Quality of Approximate Balanced Truncation}

\author{Lei-Hong Zhang%
\thanks{School of Mathematical Sciences, Soochow University, Suzhou 215006, Jiangsu, China.
        This work was
 supported in part by the National Natural Science Foundation of China NSFC-12071332, NSFC-12371380, Jiangsu Shuangchuang Project (JSSCTD202209), and Academic Degree and Postgraduate Education Reform Project of Jiangsu Province.
        Email: {\tt longzlh@suda.edu.cn}. }
\and
Ren-Cang Li%
\thanks{Department of Mathematics, University of Texas at Arlington, Arlington, TX 76019-0408, USA.
        Supported in part by NSF DMS-1719620 and DMS-2009689.
        Email: {\tt rcli@uta.edu}.}
}

\date{ }

\begin{document}

\maketitle

\begin{abstract}
Model reduction is a powerful tool in dealing with numerical simulation of large scale dynamic systems
for studying complex physical systems. Two major
types of model reduction methods for linear time-invariant dynamic systems are Krylov subspace-based methods
and balanced truncation-based methods. The methods of the second type are much more theoretically sound than the first type in that there is a fairly tight global error bound on the approximation error between the original system and
the reduced one. It is noted that the error bound
is established based upon the availability of the exact  controllability and observability Gramians. However, numerically,
the Gramians are not available and have to be numerically calculated, and for a large scale system,
a viable option is to compute low-rank approximations of the Gramians from which
an approximate balanced truncation is then performed. Hence, rigorously speaking,
the existing  global error bound is not applicable to any reduced system obtained via approximate Gramians. The goal of this paper is to address this issue by establishing global error bounds
for reduced systems via approximate balanced truncation.

\bigskip
\noindent
{\bf Keywords:}  model reduction, balanced Truncation, transfer function, controllability  Gramian, observability Gramian, Hankel singular value, low-rank approximation, error bound

\smallskip
\noindent
{\bf Mathematics Subject Classification}  78M34, 93A15, 93B40
\end{abstract}

\clearpage
\tableofcontents

\clearpage
\section{Introduction}\label{sec:intro}
Model reduction is a powerful tool in dealing with numerical simulation of large scale dynamic systems
for studying complex physical systems \cite{anto:2005,freu:2003,liba:2005}. In this paper, we are interested in the following
continuous linear time-invariant dynamic system
\begin{subequations}\label{eq:CDS}
\begin{align}
\bx'(t) &=A\,\bx(t)+ B\,\bu(t), \quad\mbox{given $\bx(0)=\bx_0$}, \label{eq:CDS-1}\\
\by(t) &=C^{\T}\,\bx(t), \label{eq:CDS-2}
\end{align}
\end{subequations}
where $\bx\,:\,t\in[0,\infty)\to\bbR^n$ is the state vector,
and $\bu\,:\,t\in[0,\infty)\to\bbR^m$ is the input, $\by\,:\,t\in[0,\infty)\to\bbR^p$ is the output, and
$A\in\bbR^{n\times n},\,B\in\bbR^{n\times m},\,C\in\bbR^{n\times p}$ are constant matrices that define the dynamic system.
In today's applications of interests, such as very large scale integration (VLSI)  circuit designs and structural dynamics,
$n$ can be up to millions \cite{ansg:2001,bai:2002,freu:2003}, but usually the dimensions of input and output vectors
are much smaller, i.e., $p,\,m\ll n$. Large $n$ can be an obstacle in practice both computationally and in memory
usage. Model reduction is then called for to overcome the  obstacle.

In a nutshell, model reduction for dynamic system \eqref{eq:CDS} seeks two matrices $X,\,Y\in\bbR^{n\times r}$ such that $Y^{\T}X=I_r$ to
reduce the system \eqref{eq:CDS} to
\begin{subequations}\label{eq:rCDS}
\begin{align}
\hat\bx_r'(t) &=A_r\,\hat\bx_r(t)+ B_r\,\bu(t), \quad\mbox{given $\hat\bx_r(0)=Y^{\T}\bx_0$}, \label{eq:rCDS-1}\\
\by(t) &=C_r^{\T}\,\hat\bx_r(t), \label{eq:rCDS-2}
\end{align}
\end{subequations}
where $A_r,\,B_r,\,C_r$ are given by
\begin{equation}\label{eq:reduced-matrices}
A_r:=Y^{\T}AX\in\bbR^{r\times r}, \quad
B_r:=Y^{\T}B\in\bbR^{r\times m}, \quad
C_r:=X^{\T}C\in\bbR^{r\times p}.
\end{equation}
Intuitively,
this reduced system \eqref{eq:rCDS} may be thought of obtaining from \eqref{eq:CDS}
by letting $\bx=X\hat\bx_r$ and performing Galerkin projection with $Y$.
The new state vector $\hat\bx_r$ is now in $\bbR^r$, a much smaller space in dimension than $\bbR^n$. In practice,
for the reduced system to be of any use, the two systems must be ``close'' in some sense.

Different model reduction methods differ in their choosing $X$ and $Y$, the projection matrices. There are two major
types: Krylov subspace-based methods \cite{bai:2002,freu:2003,liba:2005}
and balanced truncation-based methods \cite{anto:2005,guan:2004}. The methods of the first type
are computationally more efficient for large scale systems and reduced models are accurate around points where Krylov subspaces are built,
while those of the second type are theoretically sound
in that fairly tight global error bounds are known but numerically much more expensive in that
controllability and observability Gramians which are provably positive definite have to be computed at cost of $O(n^3)$ complexity.

Modern balanced truncation-based methods have improved, thanks to the discovery that the Gramians are usually
numerically low-rank \cite{ansz:2002,penz:2000,sabi:2006} and methods that compute their low-rank approximations in the factor form
\cite{liwh:2004,belt:2009}. The low-rank factors are then naturally used to compute an approximate balanced truncation.
Moments ago, we pointed out the advantage of balanced truncation-based methods in their sound global approximations guaranteed by
tight global error bounds, but these bounds are established based on exact Gramians and hence the classical exact-Gramian error bounds do not directly apply to a reduced model constructed from approximate Gramians. To the best of our knowledge, a quantitative error bound that directly relates independently prescribed approximation errors in the controllability and observability Gramians to the resulting reduced transfer function appears to be unavailable. Our aim in this paper is to address the void.

The rest of this paper is organized as follows. \Cref{sec:MOR:review} reviews the basics of
balanced truncation methods. \Cref{sec:approxBT} explains approximate balanced truncation,
when some low-rank approximations of controllability and observability Gramians, not the exact Gramians
themselves, are available. In \Cref{sec_error_apr_balance}, we establish our main results to
quantify the accuracy of the reduced model by approximate balanced reduction. We draw our conclusions
and make some remarks. Some preliminary material on subspaces of $\bbR^n$ and perturbation for Lyapunov equation are discussed in appendices.

\bigskip
\noindent
{\bf Notation.}
$\bbR^{m\times n}$  is the set of $m\times n$ real matrices,  $\bbR^n=\bbR^{n\times 1}$, and $\bbR=\bbR^1$.
$I_n\in\bbR^{n\times n}$ is
the identity matrix or simply $I$ if its size is clear from the context, and $\be_j$ is the $j$th column of $I$ of apt size.
$B^{\T}$ stands for the transpose of a matrix/vector.
$\cR(B)$ is the column subspace of $B$, spanned by its columns.
For $B\in\bbR^{m\times n}$, its singular values are
$$
\sigma_1(B)\ge\sigma_2(B)\ge\cdots\ge\sigma_k(B)\ge 0,
$$
where $k=\min\{m,n\}$, and $\sigma_{\max}(B)=\sigma_1(B)$ and $\sigma_{\min}(B)=\sigma_k(B)$.
$\|B\|_2$, $\|B\|_{\F}$, and $\|B\|_{\UI}$ are its spectral and Frobenius
norms:
$$
\|B\|_2=\sigma_1(B),\,\,
\|B\|_{\F}=\Big(\sum_{i=1}^k[\sigma_i(B)]^2\Big)^{1/2},\,\,
$$
respectively. $\|B\|_{\UI}$ is some unitarily invariant norm of $B$ \cite{lizh:2015,stsu:1990}.
For a matrix $A\in\bbR^{n\times n}$ that is known to have real eigenvalues only, $\eig(A)=\{\lambda_i(A)\}_{i=1}^n$ denotes the set of its eigenvalues (counted by multiplicities)
arranged in the decreasing order, and $\lambda_{\max}(A)=\lambda_1(A)$ and $\lambda_{\min}(A)=\lambda_n(A)$.
$A\succ 0\, (\succeq 0)$ means that it is symmetric and positive definite (semi-definite), and
accordingly
$A\prec 0\, (\preceq 0)$ if $-A\succ 0\, (\succeq 0)$.
MATLAB-like notation is used to access the entries of a matrix:
$X_{(i:j,k:\ell)}$ to denote the submatrix of a matrix $X$, consisting of the intersections of
rows $i$ to $j$ and columns $k$ to $\ell$, and when $i : j$ is replaced by $:$, it means all rows, similarly for columns.

\section{Review of balanced truncation}\label{sec:MOR:review}
In this section, we will review the balanced truncation, minimally to the point to serve our purpose in this paper.
The reader is referred to \cite{anto:2005} for a more detailed exposition.

Consider continuous linear time-invariant dynamic system \eqref{eq:CDS} and suppose that it is stable, observable and controllable
\cite{anmo:1971,zhdg:1995}.

\subsection{Quality of a reduced order model}\label{ssec:MRd-fmwk}
Suppose initially $\bx_0=0$. Applying the Laplace transformation to \eqref{eq:CDS} yields
$$
\bY(s)=\underbrace{C^{\T}(sI_n-A)^{-1}B}_{:=H(s)}\bU(s), \quad s\in\bbC,
$$
where $\bU(s)$ and $\bY(s)$ are the Laplace transformations of $\bu$ and $\by$, respectively, and $H(s)\in \bbC^{p\times m}$
is the so-called {\em transfer function\/} of system \eqref{eq:CDS}.
Conveniently, we will adopt the notation
to denote the system \eqref{eq:CDS} symbolically by
$$
\scrS=\left(\begin{array}{c|c}
               A & B \\ \hline
              \vextra C^{\T}&
              \end{array}\right)
$$
with the round bracket to distinguish it from the square bracket for matrices.
 
The $\cH_{\infty}$-norm  of the system $\scrS$  is defined as
\begin{equation}\label{eq:hankelH}
 \|H(\cdot)\|_{\cH_{\infty}}:=\sup_{\omega\in \bbR}\|H(\iota \omega)\|_2
   =\sup_{\omega\in \bbR}\sigma_{\max}(H(\iota\omega)),
\end{equation}
where $\|\cdot\|_2$ is the spectral norm of a matrix, and $\iota=\sqrt{-1}$ is the imaginary unit.

In \Cref{sec:intro}, we introduced the framework of model reduction with two matrices
$X,\,Y\in\bbR^{n\times r}$ such that $Y^{\T}X=I_r$. For the ease of our presentation going forward, we shall
rename them as $X_1,\,Y_1\in\bbR^{n\times r}$ and $Y_1^{\T}X_1=I_r$. Next we look for
$X_2,\,Y_2\in\bbR^{n\times (n-r)}$ such that
\begin{equation}\label{eq:TTinv-why}
I_n=[Y_1,Y_2]^{\T}[X_1,X_2]=\begin{bmatrix}
                       Y_1^{\T} \\
                       Y_2^{\T}
                     \end{bmatrix}[X_1,X_2]=\begin{bmatrix}
                                              Y_1^{\T}X_1 & Y_1^{\T}X_2 \\
                                              Y_2^{\T}X_1 & Y_2^{\T}X_2
                                            \end{bmatrix}.
\end{equation}
Such $X_2,\,Y_2\in\bbR^{n\times (n-r)}$  always exist by \Cref{lm:acute-basis,lm:bi-orth-basis} if
$\|\sin\Theta(\cR(X_1),\cR(Y_1))\|_2<1$. In any practical model reduction method,
only $X_1,\,Y_1$ need to be produced. That $X_2,\,Y_2$ are introduced here is only for our analysis later. Denote by
\begin{equation}\label{eq:TTinv}
T=[Y_1,Y_2]^{\T}, \quad T^{-1}=[X_1,X_2],
\end{equation}
which are consistent because of \eqref{eq:TTinv-why}.
To the original system \eqref{eq:CDS}, perform transformation: $\bx(t)=T\hat\bx(t)$, to get
\begin{subequations}\label{eq:T-CDS}
\begin{align}
\hat\bx'(t) &=\what A\,\hat\bx(t)+ \what B\,\bu(t), \quad\hat\bx(0)=T^{-1}\bx_0, \label{eq:B-CDS-1}\\
\by(t) &=\what C^{\T}\,\hat\bx(t), \label{eq:T-CDS-2}
\end{align}
where
\begin{equation}\label{eq:Ttrans}
\what A = TAT^{-1},  \quad \what B = T B, \quad\what C = T^{-\T}C,
\end{equation}
\end{subequations}
naturally partitioned as
\begin{equation}\nonumber
\what A=\kbordermatrix{ &\sss r &\sss n-r\\
        \sss r & \what A_{11} & \what A_{12} \\
        \sss n-r & \what A_{21}  & \what A_{22} },\quad
\what B=\kbordermatrix{\sss \\
        \sss r & \what B_1 \\
        \sss n-r & \what B_2},\quad
\what C=\kbordermatrix{\sss \\
        \sss r & \what C_1 \\
        \sss n-r & \what C_2}.
\end{equation}
In particular,
\begin{equation}\label{eq:reduced-matrices:BT}
\what A_{11}:=Y_1^{\T}AX_1\in\bbR^{r\times r}, \quad
\what B_1:=Y_1^{\T}B\in\bbR^{r\times m}, \quad
\what C_1:=X_1^{\T}C\in\bbR^{r\times p}.
\end{equation}
One can verify that the transfer functions of \eqref{eq:CDS} and \eqref{eq:T-CDS} are exactly the same.

In current notation, the reduced system \eqref{eq:rCDS} in \Cref{sec:intro} takes the form
\begin{subequations}\label{eq:CDS:Redu:general}
\begin{align}
\hat\bx_r'(t) &=\what A_{11}\,\hat\bx_r(t)+\what B_1\,\bu(t), \quad\hat\bx_r(0)=Y_1^{\T}\bx_0, \label{eq:CDS:Redu-1:general}\\
\by(t) &=\what C_1^{\T}\,\hat\bx_r(t), \label{eq:CDS:Redu-2:general}
\end{align}
\end{subequations}
which will be denoted in short by
$\scrS_{\rd}=\left(\begin{array}{c|c}
                 \what A_{11} & \what B_1 \\ \hline
                 \vextra\what C_1^{\T}&
                 \end{array}\right)$.
Its transfer function  is given by
\begin{equation}\label{eq:TF:Redu-2:general}
H_{\rd}(s)=\what C_1^{\T}(sI-\what A_{11})^{-1}\what B_1, \quad s\in\bbC.
\end{equation}
Naturally, we would like that the full system \eqref{eq:T-CDS} and its reduced one \eqref{eq:CDS:Redu:general} are ``close''.
One way to measure the closeness is the $\cH_{\infty}$-norm of the difference
between the two transfer functions \cite{anto:2005,zhdg:1995}:
$$
\|H(\cdot)-H_{\rd}(\cdot)\|_{\cH_{\infty}}:=\sup_{\omega\in\bbR}\|H(\iota\omega)-H_{\rd}(\iota\omega)\|_2,
$$
assuming both systems are stable, observable and controllable \cite{zhdg:1995}.
Other ways using Gramians are by Hankel norm and $\cH_2$-norm  which we will get to later.
It turns out that $
H_{\err}(s)=H(s)-H_{\rd}(s)
$
 is the transfer function of an expanded dynamic system:
\begin{subequations}\nonumber
\begin{align}
\begin{bmatrix}
 \hat\bx'(t) \\
 \hat\bx_r'(t)
\end{bmatrix}&=\begin{bmatrix}
                 \what A &  \\
                  & \what A_{11}
               \end{bmatrix}
              \begin{bmatrix}
               \hat\bx(t) \\
               \hat\bx_r(t)
              \end{bmatrix}+\begin{bmatrix}
                             \what B \\
                              \what B_1
                            \end{bmatrix}\bu(t), \quad
                            \begin{bmatrix}
                             \hat\bx(0) \\
                             \hat\bx_r(0)
                            \end{bmatrix}=
                            \begin{bmatrix}
                             T^{-1}\bx_0 \\
                             Y_1^{\T}\bx_0
                            \end{bmatrix}, \label{eq:errLTI-1}\\
\hat\by(t)&=\begin{bmatrix}
             \what C \\
              -\what C_1
            \end{bmatrix}^{\T}\begin{bmatrix}
               \hat\bx(t) \\
               \hat\bx_r(t)
              \end{bmatrix}, \label{eq:errLTI-2}
\end{align}
\end{subequations}
or in the short notation
$$
\scrS_{\err}=\left(\begin{array}{cc|c}
                     \what A  & & \what B \\
                     & \what A_{11} & \what B_1 \\ \hline
                     \vextra\what C^{\T} &  -\what C_1^{\T}&
                     \end{array}\right).
$$

The key that really determines the quality of a reduced system is the subspaces
$\cX_1:=\cR(X_1)$ and $\cY_1:=\cR(Y_1)$ as far as the transfer function \eqref{eq:TF:Redu-2:general} is concerned, as guaranteed by
the next theorem.

\begin{theorem}\label{thm:spaces=key}
Given the subspaces $\cX_1$ and $\cY_1$ of dimension $r$ such that $\|\sin\Theta(\cX_1,\cY_1)\|_2<1$,
any realizations of their basis matrices $X_1,\,Y_1\in\bbR^{n\times r}$ satisfying $Y_1^{\T}X_1=I_r$, respectively, do not
affect the transfer function \eqref{eq:TF:Redu-2:general} of reduced system \eqref{eq:CDS:Redu:general}.
\end{theorem}

\begin{proof}
Fix a pair of basis matrices $X_1,\,Y_1\in\bbR^{n\times r}$ of $\cX_1$ and $\cY_1$, respectively, such that $Y_1^{\T}X_1=I_r$.
Consider any other two  basis matrices $\check X_1,\,\check Y_1\in\bbR^{n\times r}$ of $\cX_1$ and $\cY_1$, respectively, such that $\check Y_1^{\T}\check X_1=I_r$. Then $\check X_1=X_1Z$ and $\check Y_1=Y_1W$ for some nonsingular $Z,\,W\in\bbR^{r\times r}$.
We have
$$
I_r=\check Y_1^{\T}\check X_1=(Y_1W)^{\T}(X_1Z)=W^{\T}(Y_1^{\T}X_1)Z=W^{\T}Z,
$$
implying $W^{\T}=Z^{-1}$, and
$$
\check Y_1^{\T}A\check X_1=Z^{-1}\big(Y_1^{\T}AX_1\big)Z,\,\,
\check Y_1^{\T}B=Z^{-1}\big(Y_1^{\T}B\big),\,\,
\check X_1^{\T}C=Z^{\T}\big(X_1^{\T}C\big).
$$
The transfer function associated with $\check X_1,\,\check Y_1$ is
\begin{align*}
\big(\check X_1^{\T}C\big)^{\T}\big(sI_r-\check Y_1^{\T}A\check X_1\big)^{-1}\big(\check Y_1^{\T}B\big)
  &=\big(X_1^{\T}C\big)^{\T}Z\big[Z^{-1}\big(sI_r-Y_1^{\T}AX_1\big)Z\big]^{-1}Z^{-1}\big(Y_1^{\T}B\big) \\
  &=\big(X_1^{\T}C\big)^{\T}\big(sI_r-Y_1^{\T}AX_1\big)^{-1}\big(Y_1^{\T}B\big),
\end{align*}
having nothing to do with $Z$ and $W$, as was to be shown.
\end{proof}

\subsection{Balanced truncation}\label{ssec:BT}
Balanced truncation fits into the general framework of model reduction, and thus it suffices for us to
define $X_1,\,Y_1\in\bbR^{n\times r}$ and $Y_1^{\T}X_1=I_r$ for balanced truncation accordingly.

 The controllability and observability Gramians $P$ and $Q$ are defined as the solutions to
the Lyapunov equations:
\begin{subequations}\label{eq:grams}
\begin{align}
AP+PA^{\T} & + BB^{\T}=0, \label{eq:lyap-cont}\\
A^{\T}Q+QA & + CC^{\T}=0, \label{eq:lyap-obse}
\end{align}
\end{subequations}
respectively. Under the assumption that dynamic system \eqref{eq:CDS} is stable, observable and controllable,
the Lyapunov equations have unique solutions that are positive definite, i.e., $P\succ 0$ and $Q\succ 0$. The model
order reduction based on balanced truncation \cite{anto:2005,beoc:2017} starts with a balanced transformation to
dynamic system \eqref{eq:CDS} such that both Gramians are the same and diagonal with diagonal entries being
the system's invariants, known as the {\it Hankel singular values\/} of the system.

Balanced truncation is classically introduced in the literature through
some full-rank decompositions of $P$ and $Q$:
\begin{equation}\label{eq:PQ:decomp}
P=SS^{\T}
\quad\mbox{and}\quad
Q=RR^{\T},
\end{equation}
where $S,\,R\in\bbR^{n\times n}$ and are nonsingular because $P\succ 0$ and $Q\succ 0$. But that is not necessary
in theory, namely $S,\, R$ do not have to be square, in which case both will have no fewer than $n$ columns because
the equalities in \eqref{eq:PQ:decomp} ensure $\rank(S)=\rank(P)$ and $\rank(R)=\rank(Q)$.
Later in \Cref{thm:BT-invariant}, we will show that balanced truncation is invariant
with respect to how the decompositions in \eqref{eq:PQ:decomp} are done, including non-square $S$ and $R$.
Such an invariance property is critical to our analysis.

Suppose that we have \eqref{eq:PQ:decomp} with
\begin{equation}\label{eq:SR-nonsqare}
S\in\bbR^{n\times m_1}
\quad\mbox{and}\quad
R\in\bbR^{n\times m_2}.
\end{equation}
Without loss of generality, we may assume $m_1\ge m_2$.
Let the SVD of $S^{\T}R\in\bbR^{m_1\times m_2}$ be
\begin{subequations}\label{eq:StR-SVD:BT}
\begin{equation} \label{eq:StR-SVD-1:BT}
S^{\T}R=U\Sigma V^{\T}\equiv
\kbordermatrix{ &\sss r &\sss m_1-r\\
                            & U_1 &  U_2 }\times
             \kbordermatrix{ &\sss r &\sss m_2-r\\
                           \sss r & \Sigma_1 & \\
                           \sss m_1-r & & \Sigma_2 } \times
             \kbordermatrix{ &\\
                           \sss r & V_1^{\T} \\
                           \sss m_2-r & V_2^{\T} },
\end{equation}
where
\begin{gather}
\Sigma_1=\diag(\sigma_1,\ldots,\sigma_r), \quad\Sigma_2=\begin{bmatrix}
                                                                 \diag(\sigma_{r+1},\ldots,\sigma_{m_2}) \\
                                                                 0_{(m_1-m_2)\times (m_2-r)}
                                                               \end{bmatrix},
       \label{eq:StR-SVD-2:BT} \\
\sigma_1\ge\sigma_2\ge\cdots\ge\sigma_{m_2}\ge 0. \label{eq:StR-SVD-3:BT}
\end{gather}
\end{subequations}
Only $\sigma_i$ for $1\le i\le n$ are positive and the rest are $0$.
Those $\sigma_i$ for $1\le i\le n$ are the so-called {\em Hankel singular values\/} of the system, and they are invariant with respect to different ways
of decomposing $P$ and $Q$ in \eqref{eq:PQ:decomp} with \eqref{eq:SR-nonsqare}, and, in fact, they are  the square roots of the eigenvalues of $PQ$, which are real and positive.
To see this, we note $\{\sigma_i^2\}$ are the eigenvalues of $(S^{\T}R)^{\T}(S^{\T}R)=R^{\T}SS^{\T}R=R^{\T}PR$ whose nonzero eigenvalues are the same as those
of $PRR^{\T}=PQ$.

Define
\begin{equation}\label{eq:T}
T=(\Sigma_{(1:n,1:n)})^{-1/2}V_{(:,1:n)}^{\T}R^{\T}.
\end{equation}
It can be verified that $T^{-1}=SU_{(:,1:n)}(\Sigma_{(1:n,1:n)})^{-1/2}$ because
\begin{align*}
\big[(\Sigma_{(1:n,1:n)})^{-1/2}V_{(:,1:n)}^{\T}R^{\T}\big]&\big[SU_{(:,1:n)}(\Sigma_{(1:n,1:n)})^{-1/2}\big] \\
 &=(\Sigma_{(1:n,1:n)})^{-1/2}V_{(:,1:n)}^{\T}(R^{\T}S)U_{(:,1:n)}(\Sigma_{(1:n,1:n)})^{-1/2} \\
 &=(\Sigma_{(1:n,1:n)})^{-1/2}V_{(:,1:n)}^{\T}(V\Sigma U^{\T})U_{(:,1:n)}(\Sigma_{(1:n,1:n)})^{-1/2} \\
 &=(\Sigma_{(1:n,1:n)})^{-1/2}\Sigma_{(1:n,1:n)}(\Sigma_{(1:n,1:n)})^{-1/2} \\
 &=I_n.
\end{align*}
With $T$ and $T^{-1}$, we define $\what A$, $\what B$, and $\what C$ according to \eqref{eq:Ttrans}, and, as a result,
the transformed system \eqref{eq:T-CDS}. In turn,
we have $A=T^{-1}\what AT$, $B=T^{-1}\what B$, and $C=T^{\T}\what C$. Plug these relations into \eqref{eq:grams} to get,
after simple re-arrangements,
\begin{subequations}\nonumber
\begin{align}
\what A(TPT^{\T})+(TPT^{-1})\what A^{\T} & +\what B\what B^{\T}=0, \label{eq:T-lyap-cont}\\
\what A^{\T}(T^{-\T}QT^{-1})+(T^{-\T}QT^{-1})\what A & + \what C\what C^{\T}=0, \label{eq:T-lyap-obse}
\end{align}
\end{subequations}
which are precisely the Lyapunov equations for the Gramians
\begin{equation}\label{eq:blcT}
\what P = T P T^{\T}, \quad\what Q = T^{-\T} Q T^{-1},
\end{equation}
of the transformed system \eqref{eq:T-CDS}. With the help of \eqref{eq:PQ:decomp}, \eqref{eq:StR-SVD:BT} and \eqref{eq:T},
it is not hard to verify that
\begin{equation}\nonumber
\what P=\what Q=\Sigma_{(1:n,1:n)},
\end{equation}
balancing out the Gramians.

Given integer $1\le r\le n$ (usually $r\ll n$), according to the partitions of $U$, $\Sigma$, and $V$
in \eqref{eq:StR-SVD:BT}, we write
\begin{subequations}\label{eq:X1Y1:BT}
\begin{align}
T^{-1}&=\Big[SU_1\Sigma_1^{-1/2},SU_2(\Sigma_2)_{(1:n-r,1:n-r)}^{-1/2}\Big]=:[X_1,X_2], \label{eq:X1:BT}\\
T&=\begin{bmatrix}
    \Sigma_1^{-1/2}V_1^{\T}R^{\T}\\
    (\Sigma_2)_{(1:n-r,1:n-r)}^{-1/2}V_2^{\T}R^{\T}
  \end{bmatrix}=:\begin{bmatrix}
                   Y_1^{\T} \\
                   Y_2^{\T}
                 \end{bmatrix},  \label{eq:Y1:BT}
\end{align}
\end{subequations}
leading to the reduced system \eqref{eq:CDS:Redu:general} in form but with
newly defined $X_1,\,Y_1\in\bbR^{n\times r}$ by \eqref{eq:X1Y1:BT}.
In the rest of this section, we will adopt the notations in \Cref{ssec:MRd-fmwk}
but with $X_1,\,Y_1$ given by \eqref{eq:X1Y1:BT}.

%

Balanced truncation as stated is a very expensive procedure that generates \eqref{eq:CDS:Redu:general} computationally. The computations of $P$ and $Q$ fully costs $O(n^3)$ each, by, e.g., the Bartels-Stewart algorithm
\cite{bast:1972}, decompositions $P=SS^{\T}$ and $Q=RR^{\T}$ costs  $O(n^3)$ each, and so does
computing SVD of $S^{\T}R$, not to mention $O(n^2)$ storage requirements. However, it is a well-understood method in that
the associated reduced system \eqref{eq:rCDS} inherits most important system properties of the original
system: being stable, observable and controllable, and also  there is a global error bound that
guarantees the overall quality of the reduced system.

In terms of Gramians, the Hankel norm and ${\cal H}_2$-norms
of $H(\cdot)$ are given by (e.g., \cite[Section 5.4.2]{anto:2005})
\begin{align*}
\|H(\cdot)\|_{\rm H}
   &=\sqrt{\lambda_{\max}(PQ)}=\sigma_1, \\
\|H(\cdot)\|_{{\cal H}_2}
   &=\sqrt{\tr(B^{\T}QB)}=\sqrt{\tr(C^{\T}PC)},
\end{align*}
where $\sigma_1$ is the largest Hankel singular value in \eqref{eq:StR-SVD-3:BT}. The Hankel norm is  defined as the ${\cal L}_2$-induced norm of the following Hankel operator 
$$
\mathfrak H_{\rm H}:{\cal L}_2(-\infty,0]\longrightarrow {\cal L}_2[0,\infty),~~
(\mathfrak H_{\rm H}u)(t):=
\int_{-\infty}^{0}
C^{\T}e^{A(t-\tau)}Bu(\tau)d\tau,
\qquad t\geq0.
$$  
We remark that the transformations on $P$ and $Q$ as in \eqref{eq:blcT} for any nonsingular $T$, not necessarily
the one in \eqref{eq:T}, preserve  eigenvalues of $PQ$ because
$$
\what P\what Q=(T P T^{\T})( T^{-\T} Q T^{-1})=T(PQ)T^{-1}.
$$
For the ease of future reference, we will denote
by $H_{\BT}(s)$:
\begin{equation}\label{eq:H(s):BT}
H_{\BT}(s):=\what C_1^{\T}(s I_r-\what A_{11})^{-1}\what B_1,
\end{equation}
the transfer function of the reduced system \eqref{eq:CDS:Redu:general}
with $X_1,\,Y_1\in\bbR^{n\times r}$ as in \eqref{eq:X1Y1:BT} by the balanced truncation.

The next theorem is well-known;
see, e.g., \cite[Theorem 7.9]{anto:2005}, \cite[Theorem 8.16]{zhdg:1995}.

\begin{theorem}[\cite{anto:2005,zhdg:1995}]\label{thm:BT-thy}
For $X_1$ and $Y_1$ from the balanced truncation as in \eqref{eq:X1Y1:BT}, we have
\begin{equation}\label{eq:qBT}
\sigma_{r+1}\le\|H(\cdot)-H_{\BT}(\cdot)\|_{\rm H}\le \|H(\cdot)-H_{\BT}(\cdot)\|_{\cH_{\infty}}
  \le 2\sum_{j=r+1}^n\sigma_j,
\end{equation}
where $\sigma_1\ge\sigma_2\ge\cdots\ge\sigma_n$ are the Hankel singular values of the system, i.e.,
the first $n$ singular values of $S^{\T}R$. 
\end{theorem}


One thing that is not clear yet and hasn't been drawn much attention in the literature is whether
the reduced system by the balanced truncation of order $r$ varies with the decompositions
$P=SS^{\T}$ and $Q=RR^{\T}$ which are not unique, including $S$ and $R$ that may not necessarily be square. This turns out to be an easy question to answer.

\begin{theorem}\label{thm:BT-invariant}
If $\sigma_r>\sigma_{r+1}$, then the transfer function of the reduced system \eqref{eq:CDS:Redu:general} by the balanced truncation of order $r$
is unique, regardless of any variations in the decompositions in \eqref{eq:PQ:decomp}.
\end{theorem}

\begin{proof}
We will show that the projection matrices $X_1$ and $Y_1$ defined in \eqref{eq:X1Y1:BT} are invariant with respect to
any choices of decompositions for $P$ and $Q$ of the said kind.
Suppose we have two different decompositions for each one of $P$ and $Q$
\begin{align*}
P&=SS^{\T}=\check S\check S^{\T}\quad\mbox{with}\,\, S\in\bbR^{n\times n},\, \check S\in\bbR^{n\times\check n_1}, \\
Q&=RR^{\T}=\check R\check R^{\T}\quad\mbox{with}\,\, R\in\bbR^{n\times n},\, \check R\in\bbR^{n\times\check n_2}.
\end{align*}
The idea is to show that after fixing one pair of decompositions $P=SS^{\T}$ and $Q=RR^{\T}$, $X_1$ and $Y_1$ constructed
from any other decompositions $P=\check S\check S^{\T}$ and $Q=\check R\check R^{\T}$, including nonsquare $\check S$ and $\check R$, remain the same. Evidentally $\check n_1,\,\check n_2\ge n$.

Without loss of generality, we may assume
$\check n_1\ge\check n_2$; otherwise we can append some columns of $0$ to $\check S$ from the right.

Since $\cR(P)=\cR(S)=\cR(\check S)$ and $\cR(Q)=\cR(R)=\cR(\check R)$,
there exist $W\in\bbR^{\check n_1\times n}$ and $Z\in\bbR^{\check n_2\times n}$ such that
$$
\check S=SW, \quad \check R=RZ.
$$
It can be verified that $WW^{\T}=I_n$ and $ZZ^{\T}=I_n$, i.e.,
both $W\in\bbR^{n\times \check n_1}$ and $Z\in\bbR^{n\times \check n_2}$ have orthonormal rows.
Suppose we already have the SVD of $S^{\T}R$ as
in \eqref{eq:StR-SVD:BT} with $m_1=m_2=n$.
Both $W^{\T}U\in\bbR^{\check n_1\times n}$ and $Z^{\T}V\in\bbR^{\check n_1\times n}$ have orthonormal columns.
There exist
$\check U_3\in\bbR^{\check n_1\times (\check n_1-n)}$ and
$\check V_3\in\bbR^{\check n_2\times (\check n_2-n)}$ such that
$$
[W^{\T}U, \check U_3]\in\bbR^{\check n_1\times \check n_1}
\quad\mbox{and}\quad
[Z^{\T}V, \check V_3]\in\bbR^{\check n_1\times \check n_1}
$$
are orthogonal matrices.
We have
\begin{align*}
\check S^{\T}\check R=W^{\T}S^{\T}RZ
   &=W^{\T}[U_1,U_2]
       \begin{bmatrix}
         \Sigma_1 & \\
         &\Sigma_2
       \end{bmatrix}
       \begin{bmatrix}
         V_1^{\T} \\
         V_2^{\T}
       \end{bmatrix} Z \\
   &=[W^{\T}U_1,W^{\T}U_2,\check U_3]
       \begin{bmatrix}
         \Sigma_1 & &\\
         &\Sigma_2 & \\
         & & 0_{(\check n_1-n)\times (\check n_2-n)}
       \end{bmatrix}
       \begin{bmatrix}
         (Z^{\T}V_1)^{\T} \\
         (Z^{\T}V_2)^{\T} \\
         \check V_3^{\T}
       \end{bmatrix},
\end{align*}
yielding an SVD of $\check S^{\T}\check R$, for which the corresponding projection matrices from $P=\check S\check S^{\T}$ and $Q=\check R\check R^{\T}$ are given by
$$
\check S(W^{\T}U_1)\Sigma_1^{-1/2}=SW(W^{\T}U_1)\Sigma_1^{-1/2}=SU_1\Sigma_1^{-1/2}
$$
and, similarly, $\check R(Z^{\T}V_1)\Sigma_1^{-1/2}=RV_1\Sigma_1^{-1/2}$,
yielding the same projection matrices as $X_1$ and $Y_1$ in \eqref{eq:X1Y1:BT} from
$P=SS^{\T}$ and $Q=RR^{\T}$, which in turn leads to the same reduced system \eqref{eq:CDS:Redu:general} and hence the same transfer function. Now let
$$
\check S^{\T}\check R
=[\check U_1,\check U_2,\check U_3]
       \begin{bmatrix}
         \Sigma_1 & &\\
         &\Sigma_2 & \\
         & & 0_{(\check n_1-n)\times (\check n_2-n)}
       \end{bmatrix}
       \begin{bmatrix}
         \check V_1^{\T} \\
         \check V_2^{\T} \\
         \check V_3^{\T}
       \end{bmatrix}
$$
be another SVD of $\check S^{\T}\check R$ subject to the inherent freedom in SVD, where $\check U_1\in\bbR^{\check n_1\times r}$ and $\check V_1\in\bbR^{\check n_2\times r}$. Since $\sigma_r>\sigma_{r+1}$, by
the uniqueness of singular subspaces, we know $\cR(\check U_1)=\cR(W^{\T}U_1)$ and $\cR(\check V_1)=\cR(Z^{\T}V_1)$. Therefore
$$
\cR(\check S\check U_1\Sigma_1^{-1/2})=\cR(\check S\check U_1)=\cR(\check SW^{\T}U_1)=\cR(\check SW^{\T}U_1\Sigma_1^{-1/2})=\cR(SU_1\Sigma_1^{-1/2}),
$$
and similarly, $\cR(\check R\check V_1\Sigma_1^{-1/2})=\cR(RV_1\Sigma_1^{-1/2})$, implying the same transfer function regardless of whether the reduced system is obtained by the projection matrix pair
$(X_1, Y_1)$ or by the pair $(\check S\check U_1\Sigma_1^{-1/2},\check R\check V_1\Sigma_1^{-1/2})$ by \Cref{thm:spaces=key}.
\end{proof}

\subsection{A variant of balanced truncation}\label{ssec:BT-var}
%

A distinguished feature of the transformation $T$ in \eqref{eq:X1Y1:BT} is that it makes
the transformed system \eqref{eq:T-CDS} balanced, i.e., both controllability and observability Gramians
are the same and diagonal, and so the reduced system \eqref{eq:CDS:Redu:general} is balanced, too.
But as far as just the transfer function of the reduced system is concerned, there is no need to have
$X_1$ and $Y_1$ precisely the same as the ones in \eqref{eq:X1Y1:BT} because of \Cref{thm:spaces=key}.
In fact, all that we need is to make sure
$\cR(X_1)=\cR(SU_1)$ and $\cR(Y_1)=\cR(RV_1)$, besides
$Y_1^{\T}X_1=I_r$. Specifically, we have by  \Cref{thm:spaces=key}

\begin{corollary}\nonumber
Let $\check X_1,\,\check Y_1\in\bbR^{n\times r}$ such that
\begin{equation}\label{eq:spaces=key:BT}
\cR(\check X_1)=\cR(SU_1),\,\,
\cR(\check Y_1)=\cR(RV_1),\,\,
\check Y_1^{\T}\check X_1=I_r.
\end{equation}
Then $H_{\BT}(s)\equiv \big(\check X_1^{\T}C\big)^{\T}\big(sI_r-\check Y_1^{\T}A\check X_1\big)^{-1}\big(\check Y_1^{\T}B\big)$, i.e.,
the reduced system \eqref{eq:CDS:Redu:general} with \eqref{eq:reduced-matrices:BT} obtained by replacing
$X_1,\,Y_1$ from \eqref{eq:X1Y1:BT} with $\check X_1,\,\check Y_1$ satisfying $\check Y_1^{\T}\check X_1=I_r$ has the same transfer function as the one
from the exact balanced truncation.
\end{corollary}

$X_1,\,Y_1\in\bbR^{n\times r}$ defined by \eqref{eq:X1Y1:BT} for balanced truncation are difficult to work with
in analyzing the quality of balanced truncation. Luckily, the use of transfer function for analysis
allows us to focus on the subspaces $\cR(X_1)$ and $\cR(Y_1)$.
Later, instead of the concrete forms of $X_1$ and $Y_1$ in
\eqref{eq:X1Y1:BT}, we will work with the reduced system \eqref{eq:CDS:Redu:general} with
\begin{equation}\label{eq:use-X1Y1:BT}
X_1=SU_1, \,\, Y_1=RV_1\Sigma_1^{-1}.
\end{equation}
It is not hard to verify that $\cR(X_1)=\cR(SU_1)$, $\cR(Y_1)=\cR(RV_1)$, and
$Y_1^{\T}X_1=I_r$. Effectively, in the notations of \Cref{ssec:BT} up to SVD \eqref{eq:StR-SVD:BT}, this relates to transform
the original system \eqref{eq:CDS} to \eqref{eq:T-CDS} with
\begin{equation}\label{eq:X1Y1:BT-var}
T^{-1}=\Big[SU_1,SU_2\Big], \quad
T=\begin{bmatrix}
    \Sigma_1^{-1}V_1^{\T}R^{\T}\\
    (\Sigma_2)_{(1:n-r,1:n-r)}^{-1}V_2^{\T}R^{\T}
  \end{bmatrix}.
\end{equation}
Accordingly, the Gramians for the reduced system, by \eqref{eq:blcT}, are
\begin{equation}\label{eq:PQ:BT-var}
\what P=TPT^{\T}=I_n, \quad\what Q=T^{-\T}QT^{-1}=\Sigma_{(1:n,1:n)}^2,
\end{equation}
which are not balanced, but the reduced system has the same transfer function as by the balanced truncation
with \eqref{eq:X1Y1:BT} nonetheless.

\section{Approximate balanced truncation}\label{sec:approxBT}
When $n$ is large, balanced truncation as stated is a very expensive procedure both computationally and in storage usage.
Fortunately, $P$ and $Q$ are usually numerically low-rank \cite{beto:2019,penz:2000,baes:2015,sabi:2006,ansz:2002}, which
means, $P$ and $Q$ can be very well approximated by $\wtd P=\wtd S\wtd S^{\T}$ and $\wtd Q=\wtd R\wtd R^{\T}$, respectively, where
$\wtd S\in\bbR^{n\times \wtd r_1}$ and $\wtd R\in\bbR^{n\times \wtd r_2}$ with $\wtd r_1,\,\wtd r_2\ll n$.
Naturally, we will use $\wtd S$ and $\wtd R$ to play the roles of $S$ and $R$ in \Cref{ssec:BT}.
Specifically,
a model order reduction by
 approximate balanced truncation
 goes as follows.
\begin{enumerate}
  \item compute some low-rank approximations to $P$ and $Q$ in the product form
  \begin{equation}\label{eqapproxPQ}
        P\approx \wtd P=\wtd S\wtd S^{\T}, \quad
        Q\approx \wtd Q=\wtd R\wtd R^{\T},
  \end{equation}
        where $\wtd S\in\bbR^{n\times \wtd r_1}$ and $\wtd R\in\bbR^{n\times \wtd r_2}$.
        Without loss of generality, assume $\wtd r_1\ge\wtd r_2$, for our presentation.
  \item compute SVD
        \begin{equation}\nonumber
        \wtd S^{\T}\wtd R=
            \kbordermatrix{ &\sss r &\sss \wtd r_1-r\\
                            &\wtd U_1 & \wtd U_2 }\times
             \kbordermatrix{ &\sss r &\sss \wtd r_2-r\\
                           \sss r &\wtd \Sigma_1 & \\
                           \sss \wtd r_1-r & &\wtd \Sigma_2 } \times
             \kbordermatrix{ &\\
                           \sss r &\wtd V_1^{\T} \\
                           \sss \wtd r_2-r &\wtd V_2^{\T} },
        \end{equation}
        where 
        $\wtd \Sigma_1=\diag(\wtd\sigma_1,\wtd\sigma_2,\ldots,\wtd\sigma_r)$, and
        $\wtd \Sigma_2=\begin{bmatrix}
                         \diag(\wtd\sigma_{r+1},\wtd\sigma_{r+2},\ldots,\wtd\sigma_{\wtd r_2}) \\
                         0_{(\wtd r_1-\wtd r_2)\times (\wtd r_2-r)} \\
                       \end{bmatrix}$  with
        these $\wtd\sigma_i$ arranged in the decreasing order, as in \eqref{eq:StR-SVD-3:BT} for $\sigma_i$.
  \item finally, $A$, $B$, and $C$ are reduced to
        \begin{equation}\label{eq:reduced-matrices:approxBT}
        \wtd A_{11}:=\wtd Y_1^{\T}A\wtd X_1, \quad
        \wtd B_1:=\wtd Y_1^{\T}B, \quad
        \wtd C_1:=\wtd X_1^{\T}C,
        \end{equation}
        where
        \begin{equation}\label{eq:X1Y1:approxBT}
        \wtd X_1=\wtd S\wtd U_1\wtd \Sigma_1^{-1/2}, \quad\wtd Y_1=\wtd R\wtd V_1\wtd \Sigma_1^{-1/2}.
        \end{equation}
\end{enumerate}
        It can be verified that $\wtd Y_1^{\T}\wtd X_1=I_r$.
Accordingly, we will have a reduced system
\begin{subequations}\label{eq:approxBT-CDS}
\begin{align}
\wtd\bx_r'(t) &=\wtd A_{11}\,\wtd\bx_r(t)+\wtd B_1\,\bu(t), \quad\mbox{given $\wtd\bx_r(0)=\wtd Y_1^{\T}\bx_0$}, \label{eq:approxBT-1}\\
\wtd \by(t) &=\wtd C_1^{\T}\,\wtd\bx_r(t), \label{eq:approxBT-2}
\end{align}
\end{subequations}
which will not be quite the same as \eqref{eq:rCDS}
with $X_1$ and $Y_1$ in \eqref{eq:X1Y1:BT} from the (exact) balanced truncation. The transfer function of \eqref{eq:approxBT-CDS} is
\begin{equation}\label{eq:approxH(s):BT}
\wtd H_{\BT}(s)=\wtd C_1^{\T}(sI-\wtd A_{11})^{-1}\wtd B_1, \quad s\in\bbC.
\end{equation}

One lingering question that has not been addressed in the literature is
how good reduced system \eqref{eq:approxBT-CDS} is, compared to the true reduced system of balanced truncation.
The seemingly convincing argument that if $\wtd P=\wtd S\wtd S^{\T}$ and $\wtd Q=\wtd R\wtd R^{\T}$ are sufficiently accurate then
$\wtd S^{\T}\wtd R$ should approximate $S^{\T}R$ well could be doubtful because usually $\wtd r_1,\,\wtd r_2\ll n$. A different argument may say otherwise.
In order for $\wtd P=\wtd S\wtd S^{\T}$ and $\wtd Q=\wtd R\wtd R^{\T}$ to approximate $P$ and $Q$ well, respectively,
both $\wtd S$ and $\wtd R$ must approximate the dominant components of the factors $S$ and $R$ of $P$ and $Q$ well. The problem is
$\wtd r_1,\,\wtd r_2\ll n$ here while it is possible that the dominant components of $S$ and $R$  could mismatch in forming $S^{\T}R$, i.e., in the unlucky scenario, the dominant components of $S$ match the least dominant components of $R$
in forming $S^{\T}R$ and simply extracting out the dominant components of $S$ and $R$ is not enough.
Hence it becomes critically important to provide theoretical analysis that shows the quality of
approximate balanced truncation derived from $\wtd P=\wtd S\wtd S^{\T}$ and $\wtd Q=\wtd R\wtd R^{\T}$, assuming
$\|P-\wtd P\|$ and $\|Q-\wtd Q\|$ are tiny.

By the same reasoning as we argue in Subsection~\ref{ssec:BT}, the transfer function $\wtd H_{\BT}(\cdot)$ stays
the same for any $\wtd X_1,\,\wtd Y_1\in\bbR^{n\times r}$ that satisfy
\begin{equation}\label{eq:tX1tY1:key:approxBT}
\cR(\wtd X_1)=\cR(\wtd S\wtd U_1), \quad\cR(\wtd Y_1)=\cR(\wtd R\wtd V_1)
\quad\mbox{such that}\quad
\wtd Y_1^{\T}\wtd X_1=I_r,
\end{equation}
and the  pair $(\wtd X_1,\wtd Y_1)$ in \eqref{eq:X1Y1:approxBT} is just one of many concrete pairs
that satisfy \eqref{eq:tX1tY1:key:approxBT}.
Again $\wtd X_1$ and $\wtd Y_1$ in \eqref{eq:X1Y1:approxBT}
for approximate balanced truncation are difficult to work with
in our later analysis. Luckily, we  can again focus on the subspaces $\cR(\wtd X_1)$ and $\cR(\wtd Y_1)$
because of \Cref{thm:spaces=key}.
Precisely what $\wtd X_1,\,\wtd Y_1\in\bbR^{n\times r}$ to use will be specified later in
\Cref{sec_error_apr_balance} so that they will be close to $X_1$ and $Y_1$ in \eqref{eq:use-X1Y1:BT}, respectively.

We reiterate our notations for the reduced models going forward.
\begin{itemize}
  \item $(\what A_{11},\what B_1,\what C_1)$ stands for the matrices for the reduced model \eqref{eq:CDS:Redu:general} by balanced truncation with $X_1$ and $Y_1$ in \eqref{eq:use-X1Y1:BT}. It is different from the one in the literature
      we introduced earlier with $X_1$ and $Y_1$ in \eqref{eq:X1Y1:BT}, but both share the same transfer function
      denoted by $H_{\BT}(\cdot)$.
  \item $(\wtd A_{11},\wtd B_1,\wtd C_1)$ stands for the matrices for the reduced model \eqref{eq:approxBT-CDS} by approximate balanced truncation with $\wtd X_1$ and $\wtd Y_1$ specified later in \eqref{eq:setup}.
      It is different from the one in the literature we introduced earlier with
      $\wtd X_1$ and $\wtd Y_1$ in \eqref{eq:X1Y1:approxBT}, but both share the same transfer function denoted by $\wtd H_{\BT}(\cdot)$.
\end{itemize}
In the rest of this paper,
assuming $\big\|P-\wtd P\big\|_2,\,\big\|Q-\wtd Q\big\|_2\le\epsilon$, we will
\begin{enumerate}[(i)]
  \item bound $\what A_{11}-\wtd A_{11}$, $\what B_1-\wtd B_1$, $\what C_1-\wtd C_1$ in terms of $\epsilon$, where
        $\what A_{11}$, $\what B_1$, and $\what C_1$ are from exact balanced truncation
        as in \eqref{eq:reduced-matrices:BT} with $X_1,\,Y_1$ given by \eqref{eq:use-X1Y1:BT},
        while $\wtd A_{11}$, $\wtd B_1$, and $\wtd C_1$ are from the approximate balanced truncation
        as in \eqref{eq:reduced-matrices:approxBT} with $\wtd X_1,\,\wtd Y_1\in\bbR^{n\times r}$ to be specified;
  \item bound $\big\|H_{\BT}(\cdot)-\wtd H_{\BT}(\cdot)\big\|$ and $\big\|H(\cdot)-\wtd H_{\BT}(\cdot)\big\|$ in terms of $\epsilon$
        for both $\|\cdot\|_{\rm H}$ and $\|\cdot\|_{{\cal H}_2}$.
\end{enumerate}

\section{Quality of the approximate balanced reduction}\label{sec_error_apr_balance}
The exact balanced truncation requires computing the controllability and observability Gramians $P$ and $Q$
to the working precision, performing their full-rank decompositions (such as the Cholesky decomposition) and an SVD, each of which costs
$O(n^3)$ flops. It is infeasible for large scale dynamic systems. Luckily, the numbers of columns in $B$ and $C$ are usually of $O(1)$ and
$P$ and $Q$ numerically have extremely low ranks. In practice,
due to the fast decay of the Hankel singular values $\sigma_i$ \cite{ansz:2002,baes:2015,beoc:2017,beto:2019,penz:2000},
and the fact that solving the Lyapunov equations in \eqref{eq:grams}
for the full Gramians is too expensive and storing the full Gramians
takes too much space, we can only afford to compute low-rank approximations to $P$ and $Q$
in the product form as in \eqref{eqapproxPQ} \cite{liwh:2002,sabi:2006,soan:2002}. More than that,
$\wtd P$ and $\wtd Q$ generally approach $P$ and $Q$ from below, i.e.,
\begin{subequations}\label{eq:approxPQ}
\begin{equation}\label{eq:approxPQ-1}
0\preceq \wtd P=\wtd S\wtd S^{\T}\preceq P,~~0\preceq \wtd Q=\wtd R\wtd R^{\T}\preceq Q,
\end{equation}
where $\wtd S\in\bbR^{n\times \wtd r_1}$ and $\wtd R\in\bbR^{n\times \wtd r_2}$.
This is what we will assume about $\wtd P$ amd $\wtd Q$ in the rest of this paper, besides
\begin{equation}\label{eq:approxPQ-2}
\|P-\wtd P\|_2\le \epsilon_1,\quad\|Q-\wtd Q\|_2\le \epsilon_2
\end{equation}
\end{subequations}
for some sufficiently tiny $\epsilon_1$ and $\epsilon_2$. Except their existences, exactly what $P$, $Q$ and their full-rank factors $S$ and $R$ are
 not needed in our analysis. Because of \eqref{eq:approxPQ-1}, we may write
\begin{subequations}\nonumber
\begin{align}
P&=\wtd P+EE^{\T}=[\wtd S, E][\wtd S, E]^{\T}=SS^{\T}, \label{eq:PQEF-1} \\
Q&=\wtd Q+FF^{\T}=[\wtd R, F][\wtd R, F]^{\T}=RR^{\T}, \label{eq:PQEF-2}
\end{align}
where $E\in\bbR^{n\times p_1}$ and $F\in\bbR^{n\times p_2}$ are unknown, and neither are
\begin{equation}\label{eq:tStR->SR}
S=[\wtd S, E]\in\bbR^{n\times m_1}, \quad
R=[\wtd R, F]\in\bbR^{n\times m_2},
\end{equation}
\end{subequations}
$m_1=\wtd r_1+p_1$ and $m_2=\wtd r_2+p_2$. Without loss of generality, we may assume
\begin{equation}\nonumber
m_1\ge m_2;
\end{equation}
otherwise, we simply append a few columns of $0$ to $E$.
Let
\begin{subequations}\label{eq:GtG:BT}
\begin{align}
G&:=S^{\T}R = [\wtd S, E]^{\T}[\wtd R, F]
   =\kbordermatrix{ &\sss\wtd r_2 &\sss p_2 \\
               \sss\wtd r_1 & \wtd S^{\T}\wtd R & \wtd S^{\T}F \\
               \sss p_1 & E^{\T}\wtd R & E^{\T}F }, \label{eq:GtG-1:BT} \\
\wtd G&:=\kbordermatrix{ &\sss\wtd r_2 &\sss p_2 \\
               \sss\wtd r_1 & \wtd S^{\T}\wtd R & 0 \\
               \sss p_1 & 0 & 0 }
       =G-\kbordermatrix{ &\sss\wtd r_2 &\sss p_2 \\
               \sss\wtd r_1 & 0 & \wtd S^{\T}F \\
               \sss p_1 & E^{\T}\wtd R & E^{\T}F }. \label{eq:GtG-2:BT}
\end{align}
\end{subequations}
It is reasonable to require
$$
\wtd r_i\ge r\quad\mbox{for $i=1,2$},
$$
because we are looking for balanced truncation of order $r$.
\Cref{lem:bndEF} provides some basic inequalities we need in the rest of this paper.

\begin{lemma}\label{lem:bndEF}
Suppose that \eqref{eq:approxPQ} holds. Then
\begin{alignat}{2}
\|\wtd S\|_2=\sqrt{\|\wtd P\|_2}&\le\|S\|_2=\sqrt{\|P\|_2},
            &\quad\big\|\wtd R\big\|_2=\sqrt{\|\wtd Q\|_2}&\le\|R\|_2=\sqrt{\|Q\|_2}, \label{eq:nrmSR} \\
\|E\|_2&\le \sqrt{\epsilon_1}, &\quad\|F\|_2&\le \sqrt{\epsilon_2}, \label{eq:nrmEF}
\end{alignat}
and
\begin{align}
\|\wtd G-G\|_2&=\left\|\begin{bmatrix}
                   0 &   \wtd S^{\T}F \\
                   E^{\T}\wtd R  & E^{\T}F
                 \end{bmatrix}\right\|_2  \nonumber \\
             &\le\max\left\{\sqrt{\|P\|_2\epsilon_2},\sqrt{\|Q\|_2\epsilon_1}\right\}+\sqrt{\epsilon_1\epsilon_2}=:\varepsilon.
             \label{eq:nrm(GtG)}
\end{align}
\end{lemma}

\begin{proof}
We have $\|\wtd S\|_2^2=\|\wtd S\wtd S^{\T}\|_2=\|\wtd P\|_2\le\|P\|_2=\|SS^{\T}\|_2=\|S\|_2^2$, proving the first
relation in \eqref{eq:nrmSR}.
It follows from $P-\wtd P=EE^{\T}$ that
$\|P-\wtd P\|_2=\|E\|_2^2$,
yielding the first inequality in \eqref{eq:nrmEF} upon using \eqref{eq:approxPQ}.
Similarly, we will have the second relation in \eqref{eq:nrmSR} and the second inequality in \eqref{eq:nrmEF}.


For \eqref{eq:nrm(GtG)}, we have
\begin{align*}
\left\|\begin{bmatrix}
                   0 &   \wtd S^{\T}F \\
                   E^{\T}\wtd R  & E^{\T}F
                 \end{bmatrix}\right\|_2
&\le \left\|\begin{bmatrix}
                   0 &   \wtd S^{\T}F \\
                   E^{\T}\wtd R  & 0
                 \end{bmatrix}\right\|_2+\left\|\begin{bmatrix}
                   0 &   0 \\
                   0  & E^{\T}F
                 \end{bmatrix}\right\|_2 \\
&=\max\left\{\|\wtd S^{\T}F\|_2,\|E^{\T}\wtd R\|_2\right\}+\|E^{\T}F\|_2 \\
&\le\max\left\{\sqrt{\|P\|_2\epsilon_2},\sqrt{\|Q\|_2\epsilon_1}\right\}+\sqrt{\epsilon_1\epsilon_2},
\end{align*}
as was to be shown.
\end{proof}

\begin{remark}\label{rk:nrm2use}
{\rm
Besides the spectral norm $\|\cdot\|_2$, the Frobenius
norm is another commonly used matrix norm, too. Naturally, we are wondering if we could have Frobenius-norm versions of
\eqref{eq:approxPQ-2} and \Cref{lem:bndEF}. Theoretically, it can be done, but there is one potential problem which is
that matrix dimension $n$ will show up. Here is why:
$$
\|E\|_{\F}^2\le{\sqrt{\rank(E)}}\,\|EE^{\T}\|_{\F}=\sqrt{\rank(E)}\,\|P-\wtd  P\|_{\F},
$$
and this inequality becomes an equality if all singular values of $E$ are the same. Although $\rank(E)\le n$ always, potentially
$\rank(E)=n$, bringing $n$ into the estimates here and forward. That can be an unfavorable thing to have for huge $n$.
}
\end{remark}

\subsection{Associated SVDs}
Let the SVD of $G$ in \eqref{eq:GtG-1:BT} be
\begin{subequations}\label{eq:StR-SVD:use4BT}
\begin{equation} \label{eq:StR-SVD:use4BT-1}
G= 
U\Sigma V^{\T}\equiv
\kbordermatrix{ &\sss r &\sss m_1-r\\
                            & U_1 &  U_2 }\times
             \kbordermatrix{ &\sss r &\sss m_2-r\\
                           \sss r & \Sigma_1 & \\
                           \sss m_1-r & & \Sigma_2 } \times
             \kbordermatrix{ &\\
                           \sss r & V_1^{\T} \\
                           \sss m_2-r & V_2^{\T} },
\end{equation}
where
\begin{gather}
\Sigma_1=\diag(\sigma_1,\ldots,\sigma_r), \quad\Sigma_2=\begin{bmatrix}
                                                                 \diag(\sigma_{r+1},\ldots,\sigma_{m_2}) \\
                                                                 0_{(m_1-m_2)\times (m_2-r)}
                                                               \end{bmatrix},
       \label{eq:StR-SVD:use4BT-2} \\
\sigma_1\ge\sigma_2\ge\cdots\ge\sigma_{m_2}. \label{eq:StR-SVD:use4BT-3}
\end{gather}
\end{subequations}
Despite of its large size, $G$ still has only $n$ nonzero singular values, namely $\{\sigma_i\}_{i=1}^n$, which are the Hankel singular values
of the system, and the rest of its singular values $\sigma_i=0$ for $i=n+1,\ldots, m_2$.

\begin{lemma}\label{lem:HankelSVD:approxBT}
Suppose that \eqref{eq:approxPQ-1} holds, and let the singular values of $\wtd G$ be
$$
\wtd\sigma_1\ge\wtd\sigma_2\ge\cdots\ge\wtd\sigma_{m_2}.
$$
Then $\wtd\sigma_i\le\sigma_i$ for $i=1,2,\ldots,m_2$. As a corollary, $\|\wtd G\|_2=\wtd\sigma_1\le\sigma_1$.
\end{lemma}

\begin{proof}
The nonzero singular values of $\wtd G$ are given by those of $\wtd S^{\T}\wtd R$. It suffices to
show $\wtd\sigma_i^2\le\sigma_i^2$ for $i=1,2,\ldots,\min\{\wtd r_1,\wtd r_2\}$.
Note $\wtd\sigma_i^2$ for $i=1,2,\ldots,m_1$ are the eigenvalues of
$$
\wtd G\wtd G^{\T}
    =\wtd S^{\T}\wtd R\wtd R^{\T}\wtd S=\wtd S^{\T}\wtd Q\wtd S
    \preceq\wtd S^{\T}Q\wtd S,
$$
whose nonzero eigenvalues are the same as those of
$$
Q\wtd S\wtd S^{\T}=Q\wtd P=RR^{\T}\wtd P,
$$
whose nonzero eigenvalues are the same as those of
$$
R^{\T}\wtd PR\preceq R^{\T} PR,
$$
whose nonzero eigenvalues are $\sigma_i^2$ for $i=1,2,\ldots, n$.
\end{proof}

Partition
$$
U^{\T}(\wtd G-G)V=
-U^{\T}\begin{bmatrix}
     0 & \wtd S^{\T}F \\
     E^{\T}\wtd R & E^{\T}F
   \end{bmatrix}V=\kbordermatrix{ &\sss r &\sss m_2-r \\
               \sss r & E_{11} & E_{12} \\
               \sss m_1-r & E_{21} & E_{22} }.
$$
By \Cref{lem:bndEF}, we find
\begin{equation}\label{eq:bd4Eij}
\|E_{ij}\|_2\le\|\wtd G-G\|_2\le\varepsilon\,\,\mbox{for $i,j\in\{1,2\}$},
\end{equation}
where $\varepsilon$ is defined in \eqref{eq:nrm(GtG)}.
Now we will apply
\cite[Theorem~3.1]{litz:2024} to $G,\,\wtd G$ to yield an almost SVD decomposition of $\wtd G$:
\begin{equation}\label{eq:SVD4G:approxBT}
\wtd G
  =\kbordermatrix{ &\sss r &\sss m_1-r \\
                  & \check U_1 & \check U_2}
   \times
   \kbordermatrix{ &\sss r &\sss m_2-r \\
              \sss r & \check\Sigma_1 & 0 \\
              \sss m_1-r &  0 & \check\Sigma_2}
   \times
   \kbordermatrix{ & \\
                \sss r  &\check V_1^{\T} \\
                \sss m_2-r &\check V_2^{\T} },
\end{equation}
where
\begin{subequations}\label{eq:chkUV:BT}
\begin{align}
\check U&\equiv\kbordermatrix{ &\sss r &\sss m_1-r \\
                  & \check U_1 & \check U_2}
        =[U_1,U_2]\begin{bmatrix}
                    I_r & \Gamma^{\T} \\
                    -\Gamma & I_{m_1-r}
                  \end{bmatrix}
                  \begin{bmatrix}
                    (I+\Gamma^{\T}\Gamma)^{-1/2} & 0 \\
                    0 & (I+\Gamma\Gamma^{\T})^{-1/2}
                  \end{bmatrix}, \label{eq:chkUV-1:BT}\\
\check V&\equiv\kbordermatrix{ &\sss r &\sss m_2-r \\
                  & \check V_1 & \check V_2}
        =[V_1,V_2]\begin{bmatrix}
                    I_r & -\Omega^{\T} \\
                    \Omega & I_{m_2-r}
                  \end{bmatrix}
                  \begin{bmatrix}
                    (I+\Omega^{\T}\Omega)^{-1/2} & 0 \\
                    0 & (I+\Omega\Omega^{\T})^{-1/2}
                  \end{bmatrix} \label{eq:chkUV-2:BT}
\end{align}
\end{subequations}
are two orthogonal matrices, $\Omega\in\bbR^{(m_2-r)\times r}$ and $\Gamma\in\bbR^{(m_1-r)\times r}$.

\begin{theorem}\label{thm:SVD4tG:approxBT}
Let $\varepsilon$ be as in \eqref{eq:nrm(GtG)}, and let
$$
\delta=\sigma_r-\sigma_{r+1},\,\,
\munderbar\delta=\delta-2\varepsilon,\,\,
\munderbar\sigma_r=\sigma_r-\varepsilon.
$$
If
$$
\munderbar\delta=\delta-2\varepsilon>0
\quad\mbox{and}\quad
\frac {\varepsilon^2}{\munderbar\delta^2}<\frac 14,
$$
then the following statements hold:
\begin{enumerate}[{\rm (a)}]
  \item there exist $\Omega\in\bbR^{(m_2-r)\times r}$ and $\Gamma\in\bbR^{(m_1-r)\times r}$ satisfying
        \begin{equation}\label{eq:bd4DeltaGamma}
        \max\{\|\Omega\|_2,\|\Gamma\|_2\}\le \frac {2\varepsilon}{\munderbar\delta}
        \end{equation}
        such that $\wtd G$ admits decomposition \eqref{eq:SVD4G:approxBT} with \eqref{eq:chkUV:BT};
  \item the singular values of $\wtd G$ is the multiset union of
        \begin{subequations}\label{eq:checkSigma1:approxBT}
        \begin{align}
        \check\Sigma_1&=\check U_1^{\T}\wtd G\check V_1 \nonumber \\
          &=(I+\Gamma^{\T}\Gamma)^{1/2}(\Sigma_1+E_{11} +  E_{12}\Omega)(I+\Omega^{\T}\Omega)^{-1/2}
                 \label{eq:checkSigma1-1:approxBT}\\
          &=(I+\Gamma^{\T}\Gamma)^{-1/2}(\Sigma_1+E_{11} + \Gamma^{\T} E_{21})(I+\Omega^{\T}\Omega)^{1/2},
                 \label{eq:checkSigma1-2:approxBT}
        \end{align}
        \end{subequations}
        and
        \begin{subequations}\label{eq:checkSigma2:approxBT}
        \begin{align}
        \check\Sigma_2&=\check U_2^{\T}\wtd G\check V_2 \nonumber \\
          &=(I+\Gamma\Gamma^{\T})^{1/2}(\Sigma_2+E_{22} -  E_{21}\Omega^{\T})(I+\Omega\Omega^{\T})^{-1/2}
                \label{eq:checkSigma2-1:approxBT}\\
          &=(I+\Gamma\Gamma^{\T})^{-1/2}(\Sigma_2+E_{22} - \Gamma E_{12})(I+\Omega\Omega^{\T})^{1/2};
                \label{eq:checkSigma2-2:approxBT}
        \end{align}
        \end{subequations}
  \item we have
        \begin{equation}\label{eq:checkSigma1Sigma2:approxBT}
        \sigma_{\min}(\check\Sigma_1)\ge\sigma_r-\varepsilon-\frac {2\varepsilon^2}{\munderbar\delta}, \quad
        \sigma_{\max}(\check\Sigma_2)\le\sigma_{r+1}+\varepsilon+\frac {2\varepsilon^2}{\munderbar\delta},
        \end{equation}
        where $\sigma_{\min}(\check\Sigma_1)$ and $\sigma_{\max}(\check\Sigma_2)$
        are the smallest singular value of $\check\Sigma_1$ and the largest singular value of $\check\Sigma_2$,
        respectively;
  \item if also $\varepsilon/\munderbar\delta<1/2$, then the top $r$ singular values of $\wtd G$
        are exactly the $r$ singular values of $\check\Sigma_1$,
        \begin{equation}\label{eq:checkSigma1-inproved:approxBT}
        \sigma_{\min}(\check\Sigma_1)\ge\sigma_r-\varepsilon=\munderbar\sigma_r,
        \end{equation}
        and the dominant left and right singular subspaces are spanned by the columns of
        \begin{subequations}\nonumber
        \begin{align}
        \check U_1 &=(U_1-U_2\Gamma)(I+\Gamma^{\T}\Gamma)^{-1/2}, \label{eq:checkU1V1-1:approxBT}\\
        \check V_1 &=(V_1+V_2\Omega)(I+\Omega^{\T}\Omega)^{-1/2}, \label{eq:checkU1V1-2:approxBT}
        \end{align}
        \end{subequations}
        respectively. In particular,
        \begin{subequations}\label{eq:diff(U1V1):approxBT}
        \begin{align}
        \|\check U_1-U_1\|_2
           &=\frac {\sqrt 2\,\|\Gamma\|_2}{\big[\sqrt{1+\|\Gamma\|_2^2}(\sqrt{1+\|\Gamma\|_2^2}+1)\big]^{1/2}}
                    \le\|\Gamma\|_2\le\frac {2\varepsilon}{\munderbar\delta}, \label{eq:diff(U1V1)-1:approxBT}\\
        \|\check V_1-V_1\|_2
           &=\frac {\sqrt 2\,\|\Omega\|_2}{\big[\sqrt{1+\|\Omega\|_2^2}(\sqrt{1+\|\Omega\|_2^2}+1)\big]^{1/2}}
                    \le\|\Omega\|_2\le\frac {2\varepsilon}{\munderbar\delta}, \label{eq:diff(U1V1)-2:approxBT}
        \end{align}
        \end{subequations}
        and\footnote {
             It is tempting to wonder if $\|\check \Sigma_1-\Sigma_1\|_2\le\varepsilon$, considering
             the standard perturbation result of singular values \cite[p.204]{stsu:1990}, \cite[p.21-7]{li:2014HLA}.
             Unfortunately, $\check \Sigma_1$ is unlikely diagonal.
             Another set of two inequalities for the same purpose as \eqref{eq:diff(Sigma1):approxBT}
             can be obtained as outlined in Remark~\ref{rk:diff(Sigma1):approxBT}.
             }
        \begin{subequations}\label{eq:diff(Sigma1):approxBT}
        \begin{align}
        \|\check \Sigma_1-\Sigma_1\|_2&\le \left(1+\frac {4\sigma_1}{\munderbar\delta}\right)\varepsilon, \label{eq:diff(Sigma1)-1:approxBT}\\
        \|\check \Sigma_1^{-1}-\Sigma_1^{-1}\|_2&\le\frac 1{\sigma_r\munderbar\sigma_r}
                     \left(1+\frac {4\sigma_1}{\munderbar\delta}\right){\varepsilon}. \label{eq:diff(Sigma1)-2:approxBT}
        \end{align}
        \end{subequations}
\end{enumerate}
\end{theorem}

\begin{proof}
Recall \eqref{eq:bd4Eij}.
Apply
\cite[Theorem~3.1]{litz:2024}
to $G,\,\wtd G$  with $\varepsilon_{ij}=\varepsilon$, $\delta$ and $\munderbar\delta$ here
to yield all conclusions of the theorem, except \eqref{eq:diff(U1V1):approxBT} and \eqref{eq:diff(Sigma1):approxBT}, which we will now prove.

To prove \eqref{eq:diff(U1V1)-1:approxBT}, we have
\begin{align*}
\check U_1-U_1&=U_1\big[(I+\Gamma^{\T}\Gamma)^{-1/2}-I\big]-U_2\Gamma(I+\Gamma^{\T}\Gamma)^{-1/2} \\
   &=-[U_1,U_2]\begin{bmatrix}
              I-(I+\Gamma^{\T}\Gamma)^{-1/2} \\
              \Gamma(I+\Gamma^{\T}\Gamma)^{-1/2}
            \end{bmatrix}.
\end{align*}
Let $\Gamma=Z\Xi W^{\T}$ be the SVD of $\Gamma$. We find
$$
\begin{bmatrix}
              I-(I+\Gamma^{\T}\Gamma)^{-1/2} \\
              \Gamma(I+\Gamma^{\T}\Gamma)^{-1/2}
            \end{bmatrix}=\begin{bmatrix}
                W &  \\
                 & Z
              \end{bmatrix}
            \begin{bmatrix}
              I-(I+\Xi^{\T}\Xi)^{-1/2} \\
              \Xi(I+\Xi^{\T}\Xi)^{-1/2}
            \end{bmatrix} W^{\T},
$$
where for the middle matrix on the right, $I-(I+\Xi^{\T}\Xi)^{-1/2}$ is diagonal and
$\Xi(I+\Xi^{\T}\Xi)^{-1/2}$ is leading diagonal. Hence the  singular values of the middle matrix are given by: for each singular value $\gamma$ of $\Gamma$,
\begin{align*}
\sqrt{\left(1-\frac 1{\sqrt{1+\gamma^2}}\right)^2+\left(\frac {\gamma}{\sqrt{1+\gamma^2}}\right)^2}
  &=\sqrt{2\left(1-\frac 1{\sqrt{1+\gamma^2}}\right)} \\
  &=\frac {\sqrt 2\,\gamma}{\big[\sqrt{1+\gamma^2}(\sqrt{1+\gamma^2}+1)\big]^{1/2}} \\
  &\le\gamma\le\|\Gamma\|_2.
\end{align*}
Therefore, we get
$$
\|\check U_1-U_1\|_2
            =\left\| \begin{bmatrix}
              I-(I+\Xi^{\T}\Xi)^{-1/2} \\
              \Xi(I+\Xi^{\T}\Xi)^{-1/2}
            \end{bmatrix}\right\|_2
            =\frac {\sqrt 2\,\|\Gamma\|_2}{\big[\sqrt{1+\|\Gamma\|_2^2}(\sqrt{1+\|\Gamma\|_2^2}+1)\big]^{1/2}}
            \le\|\Gamma\|_2,
$$
yielding \eqref{eq:diff(U1V1)-1:approxBT} in light of \eqref{eq:bd4DeltaGamma}.
Similarly, we have \eqref{eq:diff(U1V1)-2:approxBT}.

Finally, we prove \eqref{eq:diff(Sigma1):approxBT}. We have
\begin{align}
\check \Sigma_1-\Sigma_1
  &=\check U_1^{\T}\wtd G\check V_1- U_1^{\T}\wtd G\check V_1+U_1^{\T}\wtd G\check V_1
    -U_1^{\T} G\check V_1+U_1^{\T} G\check V_1-U_1^{\T} G V_1 \nonumber\\
  &=(\check U_1-U_1)^{\T}\wtd G\check V_1+U_1^{\T}(\wtd G-G)\check V_1+U_1^{\T} G(\check V_1-V_1).
        \label{eq:diff(Sigma1):approxBT:pf-1}
\end{align}
In light of \eqref{eq:nrm(GtG)} and \eqref{eq:diff(U1V1):approxBT}, we get
\begin{align*}
\|\check \Sigma_1-\Sigma_1\|_2
  &\le\|\check U_1-U_1\|_2\|\wtd G\|_2+\|\wtd G-G\|_2+\|G\|_2\|\check V_1-V_1\| \\
  &\le \left(1+\frac {4\sigma_1}{\munderbar\delta}\right)\varepsilon,  \\
\intertext{and}
\|\check \Sigma_1^{-1}-\Sigma_1^{-1}\|_2
  &=\|\check \Sigma_1^{-1}(\Sigma_1-\check\Sigma_1)\Sigma_1^{-1}\|_2 \\
  &\le\|\check \Sigma_1^{-1}\|_2\|\Sigma_1-\check\Sigma_1\|_2\|\Sigma_1^{-1}\|_2 \\
  &\le\frac 1{\sigma_r\munderbar\sigma_r}\left(1+\frac {4\sigma_1}{\munderbar\delta}\right)\varepsilon,
\end{align*}
completing the proof of \eqref{eq:diff(Sigma1):approxBT}.
\end{proof}

\begin{remark}\label{rk:diff(Sigma1):approxBT}
Another upper bound on $\|\check \Sigma_1-\Sigma_1\|_2$ can be obtained as follows.
Alternatively to \eqref{eq:diff(Sigma1):approxBT:pf-1}, we have
\begin{align*}
\check \Sigma_1-\Sigma_1
  &=(I+\Gamma^{\T}\Gamma)^{1/2}(\Sigma_1+E_{11} +  E_{12}\Omega)(I+\Omega^{\T}\Omega)^{-1/2}-\Sigma_1 \\
  &=(I+\Gamma^{\T}\Gamma)^{1/2}\Sigma_1(I+\Omega^{\T}\Omega)^{-1/2}-\Sigma_1 \\
  &\quad  +(I+\Gamma^{\T}\Gamma)^{1/2}(E_{11} +  E_{12}\Omega)(I+\Omega^{\T}\Omega)^{-1/2} \\
  &=(I+\Gamma^{\T}\Gamma)^{1/2}\Sigma_1(I+\Omega^{\T}\Omega)^{-1/2}-\Sigma_1(I+\Omega^{\T}\Omega)^{-1/2}
        +\Sigma_1(I+\Omega^{\T}\Omega)^{-1/2}-\Sigma_1 \\
  &\quad  +(I+\Gamma^{\T}\Gamma)^{1/2}(E_{11} +  E_{12}\Omega)(I+\Omega^{\T}\Omega)^{-1/2} \\
  &=\Big[(I+\Gamma^{\T}\Gamma)^{1/2}-I\big]\Sigma_1(I+\Omega^{\T}\Omega)^{-1/2}
        +\Sigma_1\big[(I+\Omega^{\T}\Omega)^{-1/2}-I\big] \\
  &\quad  +(I+\Gamma^{\T}\Gamma)^{1/2}(E_{11} +  E_{12}\Omega)(I+\Omega^{\T}\Omega)^{-1/2},
\end{align*}
and therefore
\begin{align}
\|\check \Sigma_1-\Sigma_1\|_2
  &\le\Big\|(I+\Gamma^{\T}\Gamma)^{1/2}-I\|_2\sigma_1+\sigma_1\big\|(I+\Omega^{\T}\Omega)^{-1/2}-I\big\|_2\nonumber \\
  &\quad  +\|(I+\Gamma^{\T}\Gamma)^{1/2}\|_2(1 +\|\Omega\|_2)\varepsilon \nonumber\\
  &\le\frac {\|\Gamma\|_2^2}{\sqrt{1+\|\Gamma\|_2^2}+1}\sigma_1
       +\sigma_1\frac {\|\Omega\|_2^2}{\sqrt{1+\|\Omega\|_2^2}(\sqrt{1+\|\Omega\|_2^2}+1)} \nonumber \\
  &\quad +\sqrt{1+\|\Gamma\|_2^2}(1+\|\Omega\|_2)\varepsilon \nonumber \\
  &\le\frac {(2\varepsilon/\munderbar\delta)^2}{\sqrt{1+(2\varepsilon/\munderbar\delta)^2}+1}
         \left(1+\frac 1{\sqrt{1+(2\varepsilon/\munderbar\delta)^2}}\right)\sigma_1
    +\sqrt{1+\left(\frac {2\varepsilon}{\munderbar\delta}\right)^2}\left(1+\frac {2\varepsilon}{\munderbar\delta}\right)\varepsilon \nonumber\\
  &\le\frac 12\left(1+\frac 1{\sqrt 2}\right)\sigma_1\left(\frac {2\varepsilon}{\munderbar\delta}\right)^2
      +2\sqrt 2\,\varepsilon.  \label{eq:diff(Sigma1):approxBT'}
\end{align}
Comparing \eqref{eq:diff(Sigma1)-1:approxBT} with \eqref{eq:diff(Sigma1):approxBT'}, we find that
both contain a term that depends only on $\varepsilon$: $\varepsilon$ in the former whereas $2\sqrt 2\,\varepsilon$ in the latter, and clearly the edge goes to \eqref{eq:diff(Sigma1)-1:approxBT} for the term,
and that both contain a term proportional to $\sigma_1$, and the edge goes to \eqref{eq:diff(Sigma1):approxBT'}
because it is $O(\sigma_1\varepsilon)$ {\em v.s.} $O(\sigma_1\varepsilon^2)$.
In the same way as how \eqref{eq:diff(Sigma1)-2:approxBT} is created,
we can create an upper bound on $\|\check \Sigma_1^{-1}-\Sigma_1^{-1}\|_2$, using
\eqref{eq:diff(Sigma1):approxBT'}, instead. Detail is omitted.
\end{remark}

As we commented on
\cite[Theorem~3.1]{litz:2024},
\eqref{eq:checkSigma1-inproved:approxBT} improves the first inequality
in \eqref{eq:checkSigma1Sigma2:approxBT}, but it relies on the latter to first establish the fact that the top $r$ singular values of $\wtd G$  are exactly the $r$ singular values of $\check\Sigma_1$.

The decomposition \eqref{eq:SVD4G:approxBT} we built for $\wtd G$ has an SVD look, but it is not an SVD because
$\check\Sigma_i$ for $i=1,2$ are not diagonal. One idea is to perform an SVD on $\check\Sigma_1$ and update $\check U_1$, $\check V_1$
accordingly to get $\wtd U_1$ and $\wtd V_1$ for the dominant left and right singular vector matrices, but
it is hard, if not impossible, to relate the resulting $\wtd U_1$ and $\wtd V_1$ to $U_1$ and $V_1$, and in return,
difficult to relate $\wtd X_1$ and $\wtd Y_1$ defined in in \eqref{eq:X1Y1:approxBT} to $X_1,\,Y_1$ defined in \eqref{eq:X1Y1:BT}. This is precisely the reason behind our previous comment
at the end of \Cref{sec:MOR:review,sec:approxBT} that $X_1,\,Y_1$ defined in \eqref{eq:X1Y1:BT}
and $\wtd X_1$ and $\wtd Y_1$ in \eqref{eq:X1Y1:approxBT} are difficult to use.
 Fortunately these concrete forms for $X_1,\,Y_1$
and $\wtd X_1$ and $\wtd Y_1$ are not essential as far the transfer functions are concerned
because of \Cref{thm:spaces=key}.
On the other hand, it is rather easy to relate $\check U_1$, $\check V_1$,
and $\check\Sigma_1$ there to $U_1$, $V_1$, and $\Sigma_1$, respectively, from the SVD of $G=S^{\T}R$.

In the rest of this paper, we will assume the following setup without explicit mentioning it:
\begin{equation}\label{eq:setup}
\framebox{
\parbox{11.5cm}{
{\bf Setup.}
Approximate Gramians $\wtd P$ and $\wtd Q$ satisfy \eqref{eq:approxPQ} such that the conditions of \Cref{thm:SVD4tG:approxBT}, including $\varepsilon/\omega<1/2$, hold.
exact balanced truncation is carried with $X_1,\,Y_1$ in \eqref{eq:use-X1Y1:BT}, while
$$ 
\wtd X_1=[\wtd S,0_{n\times (m_1-\wtd r_1)}]\check U_1, \,\,
\wtd Y_1=[\wtd R,0_{n\times (m_2-\wtd r_2)}]\check V_1\check\Sigma_1^{-1}
$$ 
are used for approximate balanced truncation. Accordingly, $\what A_{11}$, $\what B_1$, and $\what C_1$
in the reduced system \eqref{eq:CDS:Redu:general} from the exact balanced truncation
are defined by \eqref{eq:reduced-matrices:BT},
and $\wtd A_{11}$, $\wtd B_1$, and $\wtd C_1$ in the reduced system \eqref{eq:approxBT-CDS}
from approximate balanced truncation
by \eqref{eq:reduced-matrices:approxBT}.
}}
\end{equation}
$X_1,\,Y_1$ in \eqref{eq:use-X1Y1:BT} and $\wtd X_1,\,\wtd Y_1$ just intriduced,
produce different reduced models from the usual reduced models by balanced truncation in the literature, but
keep the associated transfer functions intact, nonetheless.
In particular, $\wtd X_1$ and $\wtd Y_1$ are introduced for our analysis only.
In practice, they cannot be computed because given $\wtd S$ and $\wtd R$, knowledge on what $m_1$ and $m_2$ are is not
available, a {\em priori}.

\subsection{Bounds on differences between reduced systems}\label{ssec:bd4coeffs}
In this subsection we will bound the differences
of the coefficient matrices and transfer functions between the reduced system \eqref{eq:CDS:Redu:general} from the exact balanced truncation
and \eqref{eq:approxBT-CDS} from an approximate balanced truncation.

First we will establish bounds on $\|X_1-\wtd X_1\|_2$ and $\|Y_1-\wtd Y_1\|_2$.

\begin{lemma}\label{lm:diff(X1Y1):BTaBT}
We have
\begin{subequations}\label{eq:diff(X1Y1):BTaBT}
\begin{align}
\|\wtd X_1-X_1\|_2
    &\le\sqrt{\epsilon_1}+\sqrt{\|P\|_2}\,\frac {2\varepsilon}{\munderbar\delta}=:\epsilon_x, \label{eq:diff(X1):BTaBT}\\
\|\wtd Y_1-Y_1\|_2
    &\le\frac {\sqrt{\epsilon_2}}{\sigma_r}
        +\frac {\sqrt{\|Q\|_2}}{\sigma_r}\,
        \left(1+\frac {\munderbar\delta}{2\munderbar\sigma_r}+\frac {2\sigma_1}{\munderbar\sigma_r}\right)\,
        \frac {2\varepsilon}{\munderbar\delta}
         =:\epsilon_y. \label{eq:diff(Y1):BTaBT}
\end{align}
\end{subequations}
\end{lemma}

\begin{proof}
Recall \eqref{eq:tStR->SR}. We have
\begin{align*}
\wtd X_1-X_1 &=[\wtd S,0_{n\times (m_1-\wtd r_1)}]\check U_1-S\check U_1+S\check U_1-S U_1 \\
  &=[0_{n\times \wtd r_1},-E]\check U_1+S(\check U_1-U_1),
\end{align*}
and hence, upon using \eqref{eq:nrmEF} and \eqref{eq:diff(U1V1)-1:approxBT},  and noticing
$\|S\|_2=\sqrt{\|P\|_2}$, we arrive at \eqref{eq:diff(X1):BTaBT}.
For \eqref{eq:diff(Y1):BTaBT}, we have
\begin{align*}
\wtd Y_1-Y_1
&=[\wtd R,0_{n\times (m_2-\wtd r_2)}]\check V_1\check\Sigma_1^{-1}-[\wtd R,0_{n\times (m_2-\wtd r_2)}]\check V_1\Sigma_1^{-1} \\
&\quad  +[\wtd R,0_{n\times (m_2-\wtd r_2)}]\check V_1\Sigma_1^{-1}-[\wtd R,0_{n\times (m_2-\wtd r_2)}] V_1\Sigma_1^{-1} \\
&\quad  +[\wtd R,0_{n\times (m_2-\wtd r_2)}] V_1\Sigma_1^{-1}-R V_1\Sigma_1^{-1} \\
&=[\wtd R,0_{n\times (m_2-\wtd r_2)}]\check V_1\left(\check\Sigma_1^{-1}-\Sigma_1^{-1}\right) \\
&\quad  +[\wtd R,0_{n\times (m_2-\wtd r_2)}]\left(\check V_1-V_1\right)\Sigma_1^{-1}+[0,-F] V_1\Sigma_1^{-1},
\end{align*}
and, therefore, by \Cref{lem:bndEF}, and \eqref{eq:diff(U1V1):approxBT} and \eqref{eq:diff(Sigma1):approxBT}, we get
\begin{align*}
\big\|\wtd Y_1-Y_1\big\|_2
  &\le\big\|\wtd R\big\|_2\big\|\check\Sigma_1^{-1}-\Sigma_1^{-1}\big\|_2
    +\big\|\wtd R\big\|_2\big\|\check V_1-V_1\big\|_2\big\|\Sigma_1^{-1}\big\|_2+\|F\|_2\big\|\Sigma_1^{-1}\big\|_2 \\
  &\le\sqrt{\|Q\|_2}\,\frac 1{\sigma_r\munderbar\sigma_r}\left(1+\frac {4\sigma_1}{\munderbar\delta}\right)\varepsilon
      +\sqrt{\|Q\|_2}\,\frac {2\varepsilon}{\munderbar\delta}\frac 1{\sigma_r}
      +\frac {\sqrt{\epsilon_2}}{\sigma_r},
\end{align*}
yielding \eqref{eq:diff(Y1):BTaBT}.
\end{proof}

The differences between the coefficient matrices of the two reduced systems are bounded in \Cref{thm:diff(ABC):BTaBT}
below, where the use of any unitarily invariant norm does not require additional care for proofs, and yet may be of
independent interest.

\begin{theorem}\label{thm:diff(ABC):BTaBT}
For any unitarily invariant norm $\|\cdot\|_{\UI}$, we have 
\begin{subequations}\label{eq:diff(ABC):BTaBT}
\begin{align}
\frac {\big\|\wtd A_{11}-\what A_{11}\big\|_{\UI}}{\|A\|_{\UI}}
    &\le \sqrt{\|P\|_2}\epsilon_y+\frac {\sqrt{\|Q\|_2}}{\sigma_r}\,\epsilon_x
        =:\epsilon_a  ,
        \label{eq:diff(A):BTaBT}\\
\frac {\big\|\wtd B_1-\what B_1\big\|_{\UI}}{\|B\|_{\UI}}
    &\le\epsilon_y
        =:\epsilon_b,
         \label{eq:diff(B):BTaBT} \\
\frac {\big\|\wtd C_1-\what C_1\big\|_{\UI}}{\|C\|_{\UI}}
     &\le\epsilon_x
         =:\epsilon_c  , \label{eq:diff(C):BTaBT}
\end{align}
\end{subequations}
and
\begin{subequations}\label{eq:diff(BBCC):BTaBT}
\begin{align}
\frac {\big\|\wtd B_1\wtd B_1^{\T}-\what B_1\what B_1^{\T}\big\|_{\UI}}{\|BB^{\T}\|_{\UI}}
    &\le\sqrt{\|Q\|_2}\left(\frac 1{\munderbar\sigma_r}+\frac 1{\sigma_r}\right)\,\epsilon_y
    =:\epsilon_{b2},
         \label{eq:diff(BB):BTaBT} \\
\frac {\big\|\wtd C_1\wtd C_1^{\T}-\what C_1\what C_1^{\T}\big\|_{\UI}}{\|CC^{\T}\|_{\UI}}
    &\le 2\sqrt{\|P\|_2}\,\epsilon_x
         =:\epsilon_{c2}  , \label{eq:diff(CC):BTaBT}
\end{align}
\end{subequations}
\end{theorem}

\begin{proof}
In light of  \eqref{eq:nrmSR},
it is not difficult to show that
\begin{equation}\label{eq:bd4X1Y1:BTaBT}
\|X_1\|_2\le\sqrt{\|P\|_2}, \quad
\big\|\wtd X_1\big\|_2\le\sqrt{\|P\|_2}, \quad
\|Y_1\|_2\le\frac {\sqrt{\|Q\|_2}}{\sigma_r}, \quad
\big\|\wtd Y_1\big\|_2\le\frac {\sqrt{\|Q\|_2}}{\munderbar\sigma_r},
\end{equation}
except for the last one, for which we have
$$
\big\|\wtd Y_1\big\|_2\le\big\|\wtd R\big\|_2\big\|\check\Sigma_1^{-1}\big\|_2 \\
  \le\sqrt{\|Q\|_2}\,\frac 1{\munderbar\sigma_r},
$$
which gives the last inequality in \eqref{eq:bd4X1Y1:BTaBT}.
Next we have
\begin{align}\nonumber%
\wtd A_{11}-\what A_{11}&=\wtd Y_1^{\T}A\wtd X_1- Y_1^{\T}A\wtd X_1+ Y_1^{\T}A\wtd X_1- Y_1^{\T}A X_1 \\
   &=(\wtd Y_1- Y_1)^{\T} A\wtd X_1+ Y_1^{\T}A(\wtd X_1-X_1),  \label{eq:diff(A11):BTaBT}\\
\wtd B_1-\what B_1&=\wtd Y_1^{\T}B- Y_1^{\T}B=(\wtd Y_1- Y_1)^{\T}B, \nonumber\\
\wtd C_1-\what C_1&=\wtd X_1^{\T}C- X_1^{\T}C=(\wtd X_1- X_1)^{\T}C, \nonumber
\end{align}
and
\begin{align*}
\wtd B_1\wtd B_1^{\T}-\what B_1\what B_1^{\T}
  &=\wtd Y_1^{\T}BB^{\T}\wtd Y_1- Y_1^{\T}BB^{\T}Y_1 \\
  &=\wtd Y_1^{\T}BB^{\T}\wtd Y_1-\wtd Y_1^{\T}BB^{\T} Y_1+\wtd Y_1^{\T}BB^{\T} Y_1- Y_1^{\T}BB^{\T}Y_1\\
  &=\wtd Y_1^{\T}BB^{\T}(\wtd Y_1-Y_1)+(\wtd Y_1-Y_1)^{\T}BB^{\T} Y_1, \\
\wtd C_1\wtd C_1^{\T}-\what C_1\what C_1^{\T}
  &=\wtd X_1^{\T}CC^{\T}\wtd X_1- X_1^{\T}CC^{\T}X_1  \\
  &=\wtd X_1^{\T}CC^{\T}\wtd X_1-\wtd X_1^{\T}CC^{\T} X_1+\wtd X_1^{\T}CC^{\T} X_1- X_1^{\T}CC^{\T}X_1\\
  &=\wtd X_1^{\T}CC^{\T}(\wtd X_1-X_1)+(\wtd X_1-X_1)^{\T}CC^{\T} X_1.
\end{align*}
Take any unitarily invariant norm, e.g., on \eqref{eq:diff(A11):BTaBT}, to get
$$
\big\|\wtd A_{11}-\what A_{11}\big\|_{\UI}
   \le\big\|(\wtd Y_1- Y_1)^{\T}\big\|_2\|A\|_{\UI}\big\|\wtd X_1\big\|_2+ \big\|Y_1^{\T}\big\|_2\|A\|_{\UI}\big\|\wtd X_1-X_1\big\|_2,
$$
and use \eqref{eq:diff(X1Y1):BTaBT} and \eqref{eq:bd4X1Y1:BTaBT} to conclude \eqref{eq:diff(ABC):BTaBT}.
\end{proof}

\begin{remark}\label{rk:tY1}
{\rm
With the last inequality in \eqref{eq:bd4X1Y1:BTaBT}, alternatively to \eqref{eq:diff(A11):BTaBT}, we may use
\begin{align*}
\wtd A_{11}-\what A_{11}&=\wtd Y_1^{\T}A\wtd X_1-\wtd Y_1^{\T}A X_1+\wtd Y_1^{\T}A X_1- Y_1^{\T}A X_1 \\
   &=\wtd Y_1^{\T}A(\wtd X_1-X_1)+(\wtd Y_1- Y_1)^{\T} A X_1,
\end{align*}
and get
$$
\frac {\big\|\wtd A_{11}-\what A_{11}\big\|_{\UI}}{\|A\|_{\UI}}
    \le\frac {\sqrt{\|Q\|_2}}{\munderbar\sigma_r}\,\epsilon_x+\sqrt{\|P\|_2}\,\epsilon_y,
$$
which is slightly worse than \eqref{eq:diff(A):BTaBT} because $0<\munderbar\delta<\munderbar\sigma_r=\sigma_r-\varepsilon\le\sigma_r$.
}
\end{remark}


Previously, we have introduced $H_{\BT}(\cdot)$ in \eqref{eq:H(s):BT} and $\wtd H_{\BT}(\cdot)$ in \eqref{eq:approxH(s):BT}
for the transfer functions for the reduced systems by the true and approximate balanced truncation, respectively.
Let
$$
H_{\df}(s)=H_{\BT}(s)-\wtd H_{\BT}(s),
$$
the difference between the transfer functions,
where subscript `d' is used here and in what follows to stand for `difference' between the related things from the
exact balanced truncation and its approximation.

We are interested in established bounds for
$\left\|H_{\df}(\cdot)\right\|_{\rm H}$ and
$\left\|H_{\df}(\cdot)\right\|_{{\cal H}_2}$. To this end,  we introduce
\begin{equation}\label{eq:Ks}
K_1=\int_{0}^\infty e^{  \what A_{11}t}e^{\what A_{11}^{\T}t}dt, \quad
K_2=\int_{0}^\infty e^{  \what A_{11}^{\T}t}e^{\what A_{11}t}dt,
\end{equation}
the solutions to $\what A_{11}K_1+K_1\what A_{11}^{\T}+I_r=0$ and $\what A_{11}^{\T}K_2+K_2\what A_{11}+I_r=0$, respectively, and let
\begin{equation}\label{eq:etas}
\eta_1=\|A \|_2\|K_1\|_2\epsilon_a, \quad
\eta_2=\|A \|_2\|K_2\|_2\epsilon_a.
\end{equation}
Both $K_1$ and $K_2$ are well-defined because $\what A_{11}$ is from the exact balanced truncation and hence inherits its stability from
the original state matrix $A$.

If $2\|K_2\|_2\| \widetilde A_{11}-\widehat A_{11}\|_2<1$, then by the Lyapunov stability theorem, $\wtd A_{11}$ is Hurwitz.
$H_{\df}(s)$ is the transfer function of the system\begin{equation}\label{eq:LTI:diff}
\left\{\begin{array}{lll}{\hat \bx}'_r(t)&=\what A_{11}\hat \bx_r(t)+\what B_1\bu(t), \quad\mbox{given $\hat\bx_r(0)=\wtd \bx_r(0)$},  \\
 {\wtd {\bx}}'_r(t)&=\wtd A_{11}\wtd {\bx}_r(t)+\wtd B_1\bu(t),
\\{\bz(t)}&=\what C_1^{\T}\hat \bx_r(t)- \wtd C_1^{\T}\wtd \bx_r(t),\end{array}\right.
\end{equation}
or in short,
\begin{equation}\nonumber
\scrS_{\df}=\left(\begin{array}{cc|c}
                   \what A_{11}& & \what B_1\\
                   &\wtd A_{11} &\wtd B_{1}  \\ \hline
                   \vextra\what C_1^{\T}& -\wtd C_1^{\T}&
                   \end{array}\right)
                   =:\left(\begin{array}{c|c}
                    A_{\df}& B_{\df}  \\ \hline
                   \vextra C_{\df}^{\T}&
                   \end{array}\right).
\end{equation}
Denoted by $P_{\df},\,Q_{\df}\in\bbR^{2r\times 2r}$, the controllability and observability Gramians of \eqref{eq:LTI:diff}, respectively.
They are the solutions to
\begin{subequations}\label{eq:grams:diff}
\begin{align}
A_{\df}P_{\df}+P_{\df}A_{\df}^{\T}  + B_{\df}B_{\df}^{\T}&=0, \label{eq:lyap-cont:diff}\\
A_{\df}^{\T}Q_{\df}+Q_{\df}A_{\df}  + C_{\df}C_{\df}^{\T}&=0, \label{eq:lyap-obse:diff}
\end{align}
\end{subequations}
respectively.
It is well-known that
\begin{equation}\label{eq:dfn4Hd}
\left\|H_{\df}(\cdot)\right\|_{\rm H}=\sqrt{\lambda_{\max}(P_{\df}Q_{\df})}, \quad
\left\|H_{\df}(\cdot)\right\|_{{\cal H}_2}=\sqrt{\tr(B_{\df}^{\T}Q_{\df}B_{\df})}=\sqrt{\tr(C_{\df}^{\T}P_{\df}C_{\df})}.
\end{equation}
The ${\cal H}_2$-norm of continuous system \eqref{eq:CDS} is the energy of the associated impulse response  in the time domain \cite{anto:2005}.
Our goals are then turned into estimating the largest eigenvalues of $P_{\df}Q_{\df}$ and the traces.

\begin{lemma}\label{lm:Deltas-PdQd}
If $\eta_i<1/2$ for $i=1,2$, then
\begin{equation}\label{eq:PdQd}
P_{\df}=\underbrace{\begin{bmatrix}
      I_r & I_r \\
      I_r & I_r
    \end{bmatrix}}_{=:P_0}+\underbrace{\begin{bmatrix}
                    0 & \Delta P_{12} \\
                    (\Delta P_{12})^{\T} & \Delta P_{22}
                  \end{bmatrix}}_{=:\Delta P_0}, \quad
Q_{\df}=\underbrace{\begin{bmatrix}
      \Sigma_1^2 & -\Sigma_1^2 \\
      -\Sigma_1^2 & \Sigma_1^2
    \end{bmatrix}}_{=:Q_0}+\underbrace{\begin{bmatrix}
                    0 & \Delta Q_{12} \\
                    (\Delta Q_{12})^{\T} & \Delta Q_{22}
                  \end{bmatrix}}_{=:\Delta Q_0},
\end{equation}
where $\Delta P_{ij},\,\Delta Q_{ij}\in\bbR^{r\times r}$ and satisfy
\begin{subequations}\label{eq:bd4Pd}
\begin{align}
\|\Delta P_{12}\|_2
   &\le\frac {\|K_1\|_2}{1-\eta_1}
            \left(
             \big\|\what B_1\big\|_2\|B\|_2\epsilon_b
             +\|A\|_2\epsilon_a\right)=:\xi_1, \label{eq:bd4Pd-12}\\
\|\Delta P_{22}\|_2
  &\le \frac {\|K_1\|_2}{1-2\eta_1}
       \left(\|BB^{\T}\|_2\epsilon_{b2}+2\|A\|_2\epsilon_a
       \right)=:\xi_2, \label{eq:bd4Pd-22}
\end{align}
\end{subequations}
and
\begin{subequations}\label{eq:bd4Qd}
\begin{align}
\|\Delta Q_{12}\|_2
   &\le\frac {\|K_2\|_2}{1-\eta_2}
            \left(
             \big\|\what C_1\big\|_2\|C\|_2\epsilon_c
             +\|A\|_2\epsilon_a\right)=:\zeta_1, \label{eq:bd4Qd-12}\\
\|\Delta Q_{22}\|_2
  &\le \frac {\|K_2\|_2}{1-2\eta_2}
       \left(\|CC^{\T}\|_2\epsilon_{c2}+2\|A\|_2\epsilon_a
       \right)=:\zeta_2.  \label{eq:bd4Qd-22}
\end{align}
\end{subequations}
\end{lemma}

\begin{proof}
Partition both $P_{\df},\,Q_{\df}$ as
$$
P_{\df}=\kbordermatrix{ &\sss r &\sss r\\
            \sss r & P_{11} & P_{12} \\
            \sss r & P_{12}^{\T} & P_{22} }, \quad
Q_{\df}=\kbordermatrix{ &\sss r &\sss r\\
            \sss r & Q_{11} & Q_{12} \\
            \sss r & Q_{12}^{\T} & Q_{22} }.
$$
We start by investigating $P_{\df}$ first. Blockwise, \eqref{eq:lyap-cont:diff} is equivaent to the following three equations:
\begin{subequations}\label{eq:grams:Pd}
\begin{align}
\what A_{11}P_{11}+P_{11}\what A_{11}^{\T}  + \what B_1\what B_1^{\T}&=0, \label{eq:Pd-11}\\
\what A_{11}P_{12}+P_{12}\wtd A_{11}^{\T}   + \what B_1\wtd B_1^{\T}&=0, \label{eq:Pd-12} \\
\wtd A_{11}P_{22}+P_{22}\wtd A_{11}^{\T}  + \wtd B_1\wtd B_1^{\T}&=0. \label{eq:Pd-22}
\end{align}
\end{subequations}
It follows from \Cref{ssec:BT-var} that $P_{11}=I_r$, and from \Cref{lem:sep} that both $P_{12}$ and $P_{22}$ are near $I_r$, and
therefore the form of $P_{\df}$ as in \eqref{eq:PdQd}.
Specifically, by \eqref{eq:diff(ABC):BTaBT} and \Cref{lem:sep}, we have
$$
P_{12}=I_r+\Delta P_{12}, \quad
P_{22}=I_r+\Delta P_{22}
$$
with
\begin{align*}
\|\Delta P_{12}\|_2
   &\le\frac {\|K_1\|_2}{1-\eta_1}
            \left(
             {\big\|\what B_1\big\|_2\|\wtd B_1-\what B_1\|_2}
             +\|\wtd A_{11}-\what A_{11}\|_2\right), \\
\|\Delta P_{22}\|_2
  &\le \frac {\|K_1\|_2}{1-2\eta_1}
       \left(\|\wtd B_1\wtd B_1^{\T}-\what B_1\what B_1\|_2
             +2\|\wtd A_{11}-\what A_{11}\|_2\right).
\end{align*}
They, together with \Cref{thm:diff(ABC):BTaBT}, yield \eqref{eq:bd4Pd}.

We now turn our attention to $Q_{\df}$. Blockwise, \eqref{eq:lyap-obse:diff} is equivalent to the following three equations:
\begin{subequations}\label{eq:grams:Qd}
\begin{align}
\what A_{11}^{\T}Q_{11}+Q_{11}\what A_{11}  + \what C_1\what C_1^{\T}&=0, \label{eq:Qd-11}\\
\what A_{11}^{\T}Q_{12}+Q_{12}\wtd A_{11}   - \what C_1\wtd C_1^{\T}&=0, \label{eq:Qd-12} \\
\wtd A_{11}^{\T}Q_{22}+Q_{22}\wtd A_{11}  + \wtd C_1\wtd C_1^{\T}&=0. \label{eq:Qd-22}
\end{align}
\end{subequations}
It follows from \Cref{ssec:BT-var} that $Q_{11}=\Sigma_1^2$, and from \Cref{lem:sep} that both $-Q_{12}$ and $Q_{22}$
are near $\Sigma_1^2$, and
therefore the form of $Q_{\df}$ as in \eqref{eq:PdQd}.
Specifically, by \eqref{eq:diff(ABC):BTaBT} and \Cref{lem:sep}, we have
$$
Q_{12}=-\Sigma_1^2+\Delta Q_{12}, \quad
Q_{22}=\Sigma_1^2+\Delta Q_{22}
$$
with
\begin{align*}
\|\Delta Q_{12}\|_2
   &\le\frac {\|K_2\|_2}{1-\eta_2}
            \left(
             {\big\|\what C_1\big\|_2\|\wtd C_1-\what C_1\|_2}
             +\|\wtd A_{11}-\what A_{11}\|_2\right), \\
\|\Delta Q_{22}\|_2
  &\le \frac {\|K_2\|_2}{1-2\eta_2}
       \left(\|\wtd C_1\wtd C_1^{\T}-\what C_1\what C_1\|_2
             +2\|\wtd A_{11}-\what A_{11}\|_2\right).
\end{align*}
They, together with \Cref{thm:diff(ABC):BTaBT}, yield \eqref{eq:bd4Qd}.
\end{proof}

\begin{remark}\label{rk:Hdiff}
Bounds on $\|\Delta P_{ij}\|_{\F}$ and $\|\Delta Q_{ij}\|_{\F}$ can also be established, only a little more complicated than
\eqref{eq:bd4Pd} and \eqref{eq:bd4Qd}, upon using \Cref{lem:sep} with the Frobenius norm and noticing
$$
\|I_r\|_{\F}=\sqrt r, \quad
\|\Sigma_1^2\|_{\F}=\left(\sum_{i=1}^r\sigma_i^4\right)^{1/2}\le\sqrt r\,\sigma_1^2.
$$
In fact, we will have
\begin{align*}
\|\Delta P_{12}\|_{\F}
   &\le\frac {\|K_1\|_2}{1-\eta_1}
            \left(
             {\big\|\what B_1\big\|_2\|\wtd B_1-\what B_1\|_{\F}}
             +\sqrt r\,\|\wtd A_{11}-\what A_{11}\|_2\right) \\ 
   &\le\frac {\|K_1\|_2}{1-\eta_1}
            \left(
             \big\|\what B_1\big\|_2\|B\|_{\F}\epsilon_b
             +\sqrt r\,\|A\|_2\epsilon_a\right), \\ 
\|\Delta P_{22}\|_{\F}
  &\le \frac {\|K_1\|_2}{1-2\eta_1}
       \left(\|\wtd B_1\wtd B_1^{\T}-\what B_1\what B_1\|_{\F}
             +2\sqrt r\,\|\wtd A_{11}-\what A_{11}\|_2\right) \\ 
  &\le \frac {\|K_1\|_2}{1-2\eta_1}
       \left(\|BB^{\T}\|_{\F}\epsilon_{b2}+2\sqrt r\,\|A\|_2\epsilon_a\right), \\
\|\Delta Q_{12}\|_{\F}
   &\le\frac {\|K_2\|_2}{1-\eta_2}
            \left(
             {\big\|\what C_1\big\|_2\|\wtd C_1-\what C_1\|_{\F}}
             +\|\Sigma_1^2\|_{\F}\|\wtd A_{11}-\what A_{11}\|_2\right) \\ 
   &\le\frac {\|K_2\|_2}{1-\eta_2}
            \left(
             \big\|\what C_1\big\|_2\|C\|_{\F}\epsilon_c
             +\sqrt r\,\sigma_1^2\|A\|_2\epsilon_a\right), \\ 
\|\Delta Q_{22}\|_{\F}
  &\le \frac {\|K_2\|_2}{1-2\eta_2}
       \left(\|\wtd C_1\wtd C_1^{\T}-\what C_1\what C_1\|_{\F}
             +2\|\Sigma_1^2\|_{\F}\|\wtd A_{11}-\what A_{11}\|_2\right) \\ 
  &\le \frac {\|K_2\|_2}{1-2\eta_2}
       \left(\|CC^{\T}\|_{\F}\epsilon_{c2}+2\sqrt r\,\sigma_1^2\|A\|_2\epsilon_a
       \right).
\end{align*}
But these bounds are not materially better than these straightforwardly obtained from \eqref{eq:bd4Pd} and \eqref{eq:bd4Qd}, together with
$\|M\|_{\F}\le\sqrt r\|M\|_2$ for any $M\in\bbR^{r\times r}$.
\end{remark}

\begin{theorem}\label{thm:Hdiff}
If $\eta_i<1/2$ for $i=1,2$, then
\begin{align}
\left\|H_{\df}(\cdot)\right\|_{\rm H}
  &\le\sqrt{2\sigma_1^2(\xi_1+\xi_2)+2(\zeta_1+\zeta_2)+(\xi_1+\xi_2)(\zeta_1+\zeta_2)}=:\epsilon_{{\df},{\rm H}}, \label{eq:bd4Hd}\\
\left\|H_{\df}(\cdot)\right\|_{{\cal H}_2}
  &\le\sqrt{\min\{r,m\}}\left[\sigma_1^2\big(\big\|\what B_1\big\|_2+\big\|\wtd B_1\big\|_2\big)\|B\|_2\epsilon_b\right. \nonumber \\
  &     \hspace*{2cm}\left. +\left(2\|\what B_1^{\T}\|_2\big\|\wtd B_1\big\|_2\,\zeta_1+\|\wtd B_1^{\T}\|_2^2\,\zeta_2\right)\right]^{1/2}
        =:\epsilon_{{\df},2},  \label{eq:bd4Hd:nrmH2}
\end{align} 
where  $\xi_i$ and $\zeta_i$ for $i=1,2$ are defined in \Cref{lm:Deltas-PdQd}, and
$m$ and $p$ is the numbers of columns of $B$ and $C$, respectively.
\end{theorem}

\begin{proof}
Recall \eqref{eq:PdQd}.
Noticing that
\begin{gather*}
\|P_0\|_2=2, \quad
\|Q_0\|_2=2\sigma_1^2, \quad
P_0Q_0=0, \\
\|\Delta P_0\|_2\le\|\Delta P_{12}\|_2+\|\Delta P_{22}\|_2, \quad
\|\Delta Q_0\|_2\le\|\Delta Q_{12}\|_2+\|\Delta Q_{22}\|_2,
\end{gather*}
we get
\begin{align*}
\lambda_{\max}(P_{\df}Q_{\df})
  &\le\|P_{\df}Q_{\df}\|_2 \\
  &\le\|P_0Q_0+P_0\Delta Q_0+(\Delta P_0)Q_0+(\Delta P_0)(\Delta Q_0)\|_2 \\
  &\le 2\left(\|\Delta Q_{12}\|_2+\|\Delta Q_{22}\|_2\right)
       +2\sigma_1^2\left(\|\Delta P_{12}\|_2+\|\Delta P_{22}\|_2\right) \\
  &\quad +\left(\|\Delta Q_{12}\|_2+\|\Delta Q_{22}\|_2\right)\left(\|\Delta P_{12}\|_2+\|\Delta P_{22}\|_2\right),
\end{align*}
which together with \eqref{eq:bd4Pd} and \eqref{eq:bd4Qd} lead to \eqref{eq:bd4Hd}, upon
noticing  \eqref{eq:dfn4Hd}.

Next we prove \eqref{eq:bd4Hd:nrmH2}. We claim that
\begin{align}
\tr\big(B_{\df}^{\T}Q_0B_{\df}\big)
   &\le \min\{r,m\}\,\sigma_1^2\big(\big\|\what B_1\big\|_2+\big\|\wtd B_1\big\|_2\big)\|B\|_2\epsilon_b, \label{eq:lm4H2-3}\\
|\tr(B_{\df}^{\T}[\Delta Q_0]B_{\df})|
    &\le\min\{r,m\}\left(2\|\what B_1^{\T}\|_2\big\|\wtd B_1\big\|_2\,\zeta_1+\|\wtd B_1^{\T}\|_2^2\,\zeta_2\right).
       \label{eq:lm4H2:Delta-1}
\end{align}
Note that, for any square matrix $M$,
$$
\tr(M)\le \rank(M)\,\|M\|_2.
$$
Using \Cref{thm:diff(ABC):BTaBT}, we have
\begin{align*}
\tr\big(B_{\df}^{\T}Q_0B_{\df}\big)
  &=\tr\big(\what B_1^{\T}\Sigma_1^2\what B_1-2\what B_1^{\T}\Sigma_1^2\wtd B_1+\wtd B_1^{\T}\Sigma_1^2\wtd B_1\big) \\
  &=\tr\big(\what B_1^{\T}\Sigma_1^2[\what B_1-\wtd B_1]\big)+\tr([\wtd B_1-\what B_1]^{\T}\Sigma_1^2\wtd B_1\big) \\
  &\le\min\{r,m\}(\big\|\what B_1^{\T}\Sigma_1^2[\what B_1-\wtd B_1]\big\|_2+\big\|[\wtd B_1-\what B_1]^{\T}\Sigma_1^2\wtd B_1\big\|_2) \\
  &\le\min\{r,m\}\big\|\Sigma_1^2\big\|_2\big(\big\|\what B_1\big\|_2+\big\|\wtd B_1\big\|_2\big)\|\what B_1-\wtd B_1\|_2 \\
  &\le\min\{r,m\}\,\sigma_1^2\big(\big\|\what B_1\big\|_2+\big\|\wtd B_1\big\|_2\big)\|B\|_2\epsilon_b,
\end{align*}
proving \eqref{eq:lm4H2-3}, and
\begin{align*}
|\tr\big(B_{\df}^{\T}[\Delta Q_0]B_{\df}\big)|
  &=|\tr\big(2\what B_1^{\T}[\Delta Q_{12}]\wtd B_1\big)+\wtd B_1^{\T}[\Delta Q_{22}]\wtd B_1\big)|\\
  &\le\min\{r,m\}\left(2\big\|\what B_1^{\T}[\Delta Q_{12}]\wtd B_1\big\|_2+\big\|\wtd B_1^{\T}[\Delta Q_{22}]\wtd B_1\big\|_2\right) \\
  &\le\min\{r,m\}\left(2\big\|\what B_1^{\T}\big\|_2\big\|\wtd B_1\big\|_2\|\Delta Q_{12}\|_2
              +\big\|\wtd B_1^{\T}\big\|_2^2\|\Delta Q_{22}\|_2\right),
\end{align*}
yielding \eqref{eq:lm4H2:Delta-1}. With \eqref{eq:lm4H2-3} and \eqref{eq:lm4H2:Delta-1}, we are ready to show \eqref{eq:bd4Hd:nrmH2}.
We have
\begin{align*}
\|H_{\df}(\cdot)\|_{{\cal H}_{2}}^2
    &=\tr\big(B_{\df}^{\T}Q_{\df}B_{\df}\big)=\tr\big(B_{\df}^{\T}Q_0B_{\df}\big)+\tr\big(B_{\df}^{\T}[\Delta Q_0]B_{\df}\big) \\
    &\le\min\{r,m\}\,\sigma_1^2\big(\big\|\what B_1\big\|_2+\big\|\wtd B_1\big\|_2\big)\|B\|_2\epsilon_b \\
    &\quad    +\min\{r,m\}\left(2\big\|\what B_1^{\T}\big\|_2\big\|\wtd B_1\big\|_2\,\zeta_1
               +\big\|\wtd B_1^{\T}\big\|_2^2\,\zeta_2\right),
\end{align*}
as expected.
\end{proof}

\begin{remark}\label{rk:bndappHr2}
Alternatively, basing on the second expression in \eqref{eq:dfn4Hd} for $\|H_{\df}(\cdot)\|_{{\cal H}_{2}}$, we can derive
a different bound. Similarly to \eqref{eq:lm4H2-3} and \eqref{eq:lm4H2:Delta-1}, we claim that
\begin{align}
\tr\big(C_{\df}^{\T}P_0C_{\df}\big)
    &\le \min\{r,p\}\big(\big\|\what C_1\big\|_2+\big\|\wtd C_1\big\|_2\big)\|C\|_2\epsilon_c, \label{eq:lm4H2-4} \\
|\tr\big(C_{\df}^{\T}[\Delta P_0]C_{\df}\big)|
    &\le\min\{r,p\}\left(2\big\|\what C_1^{\T}\big\|_2\big\|\wtd C_1\big\|_2\,\xi_1
                 +\big\|\wtd C_1^{\T}\big\|_2^2\,\xi_2\right). \label{eq:lm4H2:Delta-2}
\end{align}
They can be proven, analogously along the line we proved \eqref{eq:lm4H2-3} and \eqref{eq:lm4H2:Delta-1}, as follows:
\begin{align*}
\tr\big(C_{\df}^{\T}P_0C_{\df}\big)
  &\le\min\{r,p\}\big(\big\|\what C_1\big\|_2+\big\|\wtd C_1\big\|_2\big)\|\what C_1-\wtd C_1\|_2 \\
  &\le\min\{r,p\}\big(\big\|\what C_1\big\|_2+\big\|\wtd C_1\big\|_2\big)\|C\|_2\epsilon_c, \\
|\tr\big(C_{\df}^{\T}[\Delta P_0]C_{\df}\big)|
  &=|\tr\big(2\what C_1^{\T}[\Delta P_{12}]\wtd C_1\big)+\wtd C_1^{\T}[\Delta P_{22}]\wtd C_1\big)|\\
  &\le\min\{r,p\}\left(2\big\|\what C_1^{\T}[\Delta P_{12}]\wtd C_1\big\|_2+\big\|\wtd C_1^{\T}[\Delta P_{22}]\wtd C_1\big\|_2\right) \\
  &\le\min\{r,p\}\left(2\big\|\what C_1^{\T}\big\|_2\big\|\wtd C_1\big\|_2\|\Delta P_{12}\|_2
        +\big\|\wtd C_1^{\T}\big\|_2^2\|\Delta P_{22}\|_2\right).
\end{align*}
Finally,
\begin{align*}
\|H_{\df}(\cdot)\|_{{\cal H}_{2}}^2
    &=\tr\big(C_{\df}^{\T}P_{\df}C_{\df}\big)=\tr\big(C_{\df}^{\T}P_0C_{\df}\big)+\tr\big(C_{\df}^{\T}[\Delta P_0]C_{\df}\big) \\
    &\le\min\{r,p\}\big(\big\|\what C_1\big\|_2+\big\|\wtd C_1\big\|_2\big)\|C\|_2\epsilon_c \\
    &\quad +\min\{r,p\}\left(2\big\|\what C_1^{\T}\big\|_2\big\|\wtd C_1\big\|_2\,\xi_1
           +\big\|\wtd C_1^{\T}\big\|_2^2\,\xi_2\right),
\end{align*}
yielding a different $\epsilon_{{\df},2}$ from the one in \eqref{eq:bd4Hd:nrmH2}. It is not clear which one is smaller.
\end{remark}

Norms of the coefficient matrices for the reduced systems appear in the bounds in \Cref{thm:Hdiff}.
They can be replaced by the norms of the corresponding coefficient matrices for the original system with the help of
the next lemma.

\begin{lemma}\label{lm:nrms-approxBT}
For any unitarily invariant norm $\|\cdot\|_{\UI}$, we have 
\begin{subequations}\label{eq:nrms-approxBT}
\begin{alignat}{2}
\frac {\|\what A_{11}\|_{\UI}}{\|A\|_{\UI}}&\le\frac {\sqrt{\|P\|_2\|Q\|_2}}{\sigma_r}, \quad
    &\frac {\|\wtd A_{11}\|_{\UI}}{\|A\|_{\UI}}&\le\frac {\sqrt{\|P\|_2\|Q\|_2}}{\sigma_r}+\epsilon_a, \label{eq:nrms-approxBT-A}\\
\frac {\big\|\what B_1\big\|_{\UI}}{\|B\|_{\UI}}&\le\frac {\sqrt{\|Q\|_2}}{\sigma_r}, \quad
    &\frac {\big\|\wtd B_1\big\|_{\UI}}{\|B\|_{\UI}}&\le\frac {\sqrt{\|Q\|_2}}{\sigma_r}+\epsilon_b, \label{eq:nrms-approxBT-B}\\
\frac {\big\|\what C_1\big\|_{\UI}}{\|C\|_{\UI}}&\le\sqrt{\|P\|_2}, \quad
    &\frac {\big\|\wtd C_1\big\|_{\UI}}{\|C\|_{\UI}}&\le\sqrt{\|P\|_2}+\epsilon_c, \label{eq:nrms-approxBT-C}
\end{alignat}
\end{subequations}
where $\epsilon_a$,   $\epsilon_b$, and $\epsilon_c$ are as in \eqref{eq:diff(ABC):BTaBT}.
\end{lemma}

\begin{proof}
We have by \eqref{eq:use-X1Y1:BT}
\begin{align*}
\big\|\what A_{11}\big\|_{\UI} =\big\|Y_1^{\T}A X_1\big\|_{\UI}&\le\big\|Y_1^{\T}\big\|_2\|A\|_{\UI}\|X_1\|_2 \\
          &\le\|R\|_2\big\|\Sigma_1^{-1}\big\|_2\|A\|_{\UI}\|S\|_2 \\
          &=\frac {\sqrt{\|P\|_2\|Q\|_2}}{\sigma_r}\,\|A\|_{\UI}, \\
\big\|\what B_1\big\|_{\UI} =\|Y_1^{\T}B\|_{\UI}&\le\frac {\sqrt{\|Q\|_2}}{\sigma_r}\,\|B\|_{\UI}, \\
\big\|\what C_1\big\|_{\UI} =\|X_1^{\T}C\|_{\UI}&\le\sqrt{\|P\|_2}\,\|C\|_{\UI},
\end{align*}
Therefore
$$
\big\|\wtd A_{11}\big\|_{\UI}\le\big\|\what A_{11}\big\|_{\UI}+\big\|\wtd A_{11}-\what A_{11}\big\|_{\UI}
          \le\left(\frac {\sqrt{\|P\|_2\|Q\|_2}}{\sigma_r}+\epsilon_a\right)\|A\|_{\UI},
$$
and similarly for $\big\|\wtd B_1\big\|_{\UI}$ and $\big\|\wtd C_1\big\|_{\UI}$. The proofs of the other two inequalities are similar.
\end{proof}

\subsection{Transfer function for approximate balanced truncation}\label{ssec:bd4trfun}
In this subsection, we establish bounds to measure the quality of the reduced system \eqref{eq:approxBT-CDS} from
approximate balanced truncation as an approximation to the original system
\eqref{eq:CDS}. Even though the projection matrices $\wtd X_1,\,\wtd Y_1$ we used
for approximate balanced truncation are different from the ones in practice, the transfer function
as a result remains the same, nonetheless. Therefore our bounds
are applicable in real applications. These bounds are the immediate consequences of
\Cref{thm:Hdiff} upon using
$$
\big\|H(\cdot)-\wtd H_{\BT}(\cdot)\big\|\le
    \big\|H(\cdot)-H_{\BT}(\cdot)\big\|+ \big\|H_{\BT}(\cdot)-\wtd H_{\BT}(\cdot)\big\|
$$
for $\|\cdot\|=\|\cdot\|_{\rm H}$ and $\|\cdot\|_{{\cal H}_2}$.

\begin{theorem}\label{thm:main}
Under the conditions of \Cref{thm:Hdiff}, we have
\begin{align}
\big\|H(\cdot)-\wtd H_{\BT}(\cdot)\big\|_{\rm H}
    &\le \big\|H(\cdot)-  H_{\BT}(\cdot)\big\|_{\rm H}+  \epsilon_{{\df},{\rm H}}, \label{eq:bndappHrinft} \\
\big\|H(\cdot)-\wtd H_{\BT}(\cdot)\big\|_{{\cal H}_{2}}
    &\le \big\|H(\cdot)- H_{\BT}(\cdot)\big\|_{{\cal H}_{2}}+ \epsilon_{{\df},2},
                  \label{eq:bndappHr}
\end{align}
where $\epsilon_{{\df},{{\rm H}}}$ and $\epsilon_{{\df},2}$ are as in \Cref{thm:Hdiff}.
\end{theorem}

An immediate explanation to both inequalities  \eqref{eq:bndappHrinft} and \eqref{eq:bndappHr} is that
the reduced system \eqref{eq:approxBT-CDS} from the approximate balanced reduction as an approximation to the original system
\eqref{eq:CDS}
is worse than the one from the true balanced reduction by no more than
$\epsilon_{{\df},{\rm H}}$ and $\epsilon_{{\df},2}$ in terms of the $\rm Hankel$- and
${\cal H}_2$-norm, respectively. Both $\epsilon_{{\df},{\rm H}}$ and $\epsilon_{{\df},2}$ can be traced back to
the initial approximation errors $\epsilon_1$ and $\epsilon_2$ in the computed Gramians as
specified in \eqref{eq:approxPQ} albeit complicatedly. To better understand what $\epsilon_{{\df},{\rm H}}$ and $\epsilon_{{\df},2}$
are in terms of $\epsilon_1$ and $\epsilon_2$, we summarize all quantities that lead to them, up to the first order
in
$$
\epsilon_{\app}:=\max\{\epsilon_1,\epsilon_2\}.
$$
Then $\varepsilon\le\rho\sqrt{\epsilon_{\app}}+\epsilon_{\app}$ in \eqref{eq:nrm(GtG)}.
Let $\rho=\max\big\{\sqrt{\|P\|_2},\sqrt{\|Q\|_2}\big\}$. We have
\begin{alignat}{2}
\epsilon_x&\le\left(1+\frac {2\rho^2}{\delta}\right)
           \sqrt{\epsilon_{\app}}+O(\epsilon_{\app}),
           &&\quad(\mbox{see \eqref{eq:diff(X1Y1):BTaBT}}) \nonumber \\
\epsilon_y&\le\frac 1{\sigma_r} \left[1+\left(1+\frac {\delta}{2\sigma_r}+\frac {2\sigma_1}{\sigma_r}\right)
          \,\frac {2\rho^2}{\delta}\right]
           \sqrt{\epsilon_{\app}}+O(\epsilon_{\app}),
           &&\quad(\mbox{see \eqref{eq:diff(X1Y1):BTaBT}}) \nonumber \\
\epsilon_a&\le \frac {\rho}{\sigma_r}\epsilon_x+\rho\epsilon_y,
           &&\quad(\mbox{see \eqref{eq:diff(ABC):BTaBT}}) \nonumber \\
\epsilon_b&=\epsilon_y,
           &&\quad(\mbox{see \eqref{eq:diff(ABC):BTaBT}}) \nonumber \\
\epsilon_c&=\epsilon_x,
           &&\quad(\mbox{see \eqref{eq:diff(ABC):BTaBT}}) \nonumber \\
\epsilon_{b2}&=\frac {2\rho}{\sigma_r}\,\epsilon_y+O(\epsilon_{\app}),
           &&\quad(\mbox{see \eqref{eq:diff(BBCC):BTaBT}}) \nonumber \\
\epsilon_{c2}&=2\rho\epsilon_x,
           &&\quad(\mbox{see \eqref{eq:diff(BBCC):BTaBT}}) \nonumber \\
\xi_1&\le\|K_1\|_2\left(\|A\|_2\epsilon_a+\frac {\rho}{\sigma_r}\|B\|_2^2\epsilon_b\right)+O(\epsilon_{\app}),
           &&\quad(\mbox{see \eqref{eq:bd4Pd}}) \nonumber \\
\xi_2&\le\|K_1\|_2\left(2\|A\|_2\epsilon_a+\|B\|_2^2\epsilon_{b2}\right)+O(\epsilon_{\app}),
           &&\quad(\mbox{see \eqref{eq:bd4Pd}}) \nonumber \\
\zeta_1&\le\|K_2\|_2\left(\|A\|_2\epsilon_a+ {\rho}\|C\|_2^2\epsilon_c\right)+O(\epsilon_{\app}),
           &&\quad(\mbox{see \eqref{eq:bd4Qd}}) \nonumber \\
\zeta_2&\le\|K_2\|_2\left(2\|A\|_2\epsilon_a+ \|C\|_2^2\epsilon_{c2}\right)+O(\epsilon_{\app}),
           &&\quad(\mbox{see \eqref{eq:bd4Qd}}) \nonumber \\
\epsilon_{{\df},{{\rm H}}}&=\sqrt{2\sigma_1^2(\xi_1+\xi_2)+2(\zeta_1+\zeta_2)}+O(\sqrt{\epsilon_{\app}}),
           &&\quad(\mbox{see \eqref{eq:bd4Hd}}) \nonumber \\
\epsilon_{{\df},2}
        &\le\sqrt{\min\{r,m\}}\|B\|_2\left[2\sigma_1^2\frac{\rho}{\sigma_r}\epsilon_b
          +\left(\frac {\rho}{\sigma_r}\right)^2\big(2\zeta_1+\zeta_2\big)\right]^{1/2}+O(\sqrt{\epsilon_{\app}}).
           &&\quad(\mbox{see \eqref{eq:bd4Hd:nrmH2}})\nonumber
\end{alignat}
Alternatively, for $\epsilon_{{\df},2}$, also by \Cref{rk:bndappHr2}
$$
\epsilon_{{\df},2}
        \le\sqrt{\min\{r,p\}}\|C\|_2\left[2\rho\epsilon_b
          +\rho^2\big(2\xi_1+\xi_2\big)\right]^{1/2}+O(\sqrt{\epsilon_{\app}}).
$$
It can be seen that both $\epsilon_{{\df},{\rm H}}$ and $\epsilon_{{\df},2}$ are of $O(\epsilon_{\app}^{1/4})$, pretty disappointing.

\section{Concluding Remarks}
For a continuous linear time-invariant dynamic system, the existing global error bound
that bounds the error between a reduced model via balanced truncation and the original
dynamic system assumes that the reduced model is constructed from two exact controllability
and observability Gramians. But in practice, the Gramians are usually approximated by
some computed low-rank approximations, especially when the original dynamic system is large scale. Thus, rigorously speaking, the existing global error bound,
although indicative about the accuracy in the reduced system, is not really applicable.
In this paper, we perform an error analysis, assuming the reduced model is constructed
from two low-rank approximations of the Gramians, making up the deficiency in the current theory for
measuring the quality of the reduced model obtained by approximate balanced truncation.
Error bounds have been obtained for the purpose.

So far, we have been focused on continuous linear time-invariant dynamic systems, but our techniques
should be extendable to discrete time-invariant dynamic systems without much difficulty.

Throughout this paper, our presentation is restricted to the real number field $\bbR$.
This restriction is more for simplicity and clarity than
the capability of our techniques.
In fact, our approach can be straightforwardly modified to cover the complex number case: replace all transposes
$(\cdot)^{\T}$ of vectors/matrices with   conjugate transposes $(\cdot)^{\HH}$.

\clearpage
\noindent   {\Large {\bf Appendix}}
\appendix
\section{Some results on subspaces}
Consider two subspaces ${\cal U}$ and $\wtd {\cal U}$ with dimension $r$ of ${\mathbb R}^n$ and
let $U\in{\mathbb R}^{n\times r}$ and $\wtd U\in{\mathbb R}^{n\times r}$ be orthonormal basis matrices of
${\cal U}$ and ${\cal U}$, respectively, i.e.,
$$
U^{\T}U=I_r,\,{\cal U}={\cal R}(U), \quad\mbox{and}\quad
\wtd U^{\T}\wtd U=I_r,\,\wtd{\cal U}={\cal R}(\wtd U),
$$
and denote by $\tau_j$ for $1\le j\le r$ in the descending order, i.e.,
$\tau_1\ge\cdots\ge\tau_r$, the singular values of $\wtd U^{\T}U$.
The $r$ {\em canonical angles $\theta_j({\cal U},{\cal \wtd U})$
between ${\cal U}$ to ${\cal \wtd U}$\/} are defined
by
\begin{equation}\nonumber
0\le\theta_j({\cal U},{\cal \wtd U}):=\arccos\tau_j\le\frac {\pi}2\quad\mbox{for $1\le j\le r$}.
\end{equation}
They are in the ascending order, i.e., $\theta_1({\cal U},{\cal \wtd U})\le\cdots\le\theta_r({\cal U},{\cal \wtd U})$. Set
\begin{equation}\nonumber
\Theta({\cal U},{\cal \wtd U})=\diag(\theta_1({\cal U},{\cal \wtd U}),\ldots,\theta_r({\cal U},{\cal \wtd U})).
\end{equation}
It can be seen that these angles are independent of the orthonormal basis matrices $U$ and $\wtd U$ which are not unique.

We sometimes place a matrix in one of or both
arguments of $\theta_j(\,\cdot\,,\,\cdot\,)$ and $\Theta(\,\cdot\,,\,\cdot\,)$ with an understanding that it is about
the subspace spanned by the columns of the matrix argument.

It is known that $\|\sin\Theta({\cal U},{\cal \wtd U})\|_2$
defines a distance metric between ${\cal U}$ and ${\cal \wtd U}$ \cite[p.95]{sun:1987}.
 
The next two lemmas and their proofs are about how to pick up two bi-orthogonal basis matrices of two subspaces with acute canonical angles. The results  provide a foundation to some of our argument in the paper.

\begin{lemma}\label{lm:acute-basis}
Let $\cX_1$ and $\cY_1$ be two subspaces with dimension $r$ of ${\mathbb R}^n$. Then
$$\|\sin\Theta(\cX_1,\cY_1)\|_2<1$$ if and only if $Y_1^{\T}X_1$ is nonsingular for any two basis matrices
$X_1,\,Y_1\in\bbR^{n\times r}$  of
 $\cX_1$ and $\cY_1$, respectively.
\end{lemma}

\begin{proof}
Suppose that $\|\sin\Theta(\cX_1,\cY_1)\|_2<1$, and let $X_1,\,Y_1\in\bbR^{n\times r}$ be basis matrices of
 $\cX_1$ and $\cY_1$, respectively. Then
\begin{equation}\label{eq:XY2UU}
U=X_1(X_1^{\T}X_1)^{-1/2}, \quad\wtd U=Y_1(Y_1^{\T}Y_1)^{-1/2}
\end{equation}
are two orthonormal basis matrices of
$\cX_1$ and $\cY_1$, respectively. The singular values of $\wtd U^{\T}U$ are $\cos\theta_j(\cX_1,\cY_1)$ for $1\le j\le r$
which are positive because $\|\sin\Theta(\cX_1,\cY_1)\|_2<1$, and hence $\wtd U^{\T}U$ is nonsingular, and since
\begin{equation}\label{eq:XY2UU'}
\wtd U^{\T}U=(Y_1^{\T}Y_1)^{-1/2}Y_1^{\T}X_1(X_1^{\T}X_1)^{-1/2},
\end{equation}
$Y_1^{\T}X_1$ is nonsingular.

Conversely, let $X_1,\,Y_1\in\bbR^{n\times r}$ be basis matrices of
$\cX_1$ and $\cY_1$, respectively, and suppose that $Y_1^{\T}X_1$ is nonsingular. Set $U$ and $\wtd U$ as in
\eqref{eq:XY2UU}. Then $\wtd U^{\T}U$ is nonsingular by \eqref{eq:XY2UU'}, which means $\cos\theta_j(\cX_1,\cY_1)>0$ for $1\le j\le r$,
implying
$$
\|\sin\Theta(\cX_1,\cY_1)\|_2=\max_j\sqrt{1-\cos^2\theta_j(\cX_1,\cY_1)}<1,
$$
as was to be shown.
\end{proof}

\begin{lemma}\label{lm:bi-orth-basis}
Let $\cX_1$ and $\cY_1$ be two subspaces with dimension $r$ of ${\mathbb R}^n$ and suppose that
$\|\sin\Theta(\cX_1,\cY_1)\|_2<1$, i.e.,  the canonical angles between the two subspaces are acute.
\begin{enumerate}[{\rm (a)}]
  \item There exist basis matrices $X_1,\,Y_1\in\bbR^{n\times r}$ of $\cX_1$ and $\cY_1$, respectively, such that
        $Y_1^{\T}X_1=I_r$;
  \item Given a basis matrix $X_1\in\bbR^{n\times r}$ of $\cX_1$, there exists a basis matrix $Y_1\in\bbR^{n\times r}$ of $\cY_1$ such that $Y_1^{\T}X_1=I_r$;
  \item Given basis matrices $X_1,\,Y_1\in\bbR^{n\times r}$ of $\cX_1$ and $\cY_1$, respectively, such that
        $Y_1^{\T}X_1=I_r$, there exist matrices $X_2,\,Y_2\in\bbR^{n\times (n-r)}$ such that
        $$
        [Y_1,Y_2]^{\T}[X_1,X_2]=\begin{bmatrix}
                               Y_1^{\T}X_1 & Y_1^{\T}X_2 \\
                               Y_2^{\T}X_1 & Y_2^{\T}X_2
                             \end{bmatrix}=I_n.
        $$
\end{enumerate}
\end{lemma}

\begin{proof}
For item (a), first we pick two orthonormal basis matrices $U,\,\wtd U\in\bbR^{n\times r}$ of $\cX_1$ and $\cY_1$, respectively. The assumption $\|\sin\Theta(\cX_1,\cY_1)\|_2<1$ implies that the singular values of $\wtd U^{\T}U$ are $\cos\theta_j(\cX_1,\cY_2)$ for $1\le j\le r$ are positive, and hence $\wtd U^{\T}U$ is nonsingular. Now
take $X_1=U(\wtd U^{\T}U)^{-1}$ and $Y_1=\wtd U$.

For item (b), we note $U=X_1(X_1^{\T}X_1)^{-1/2}$ is an orthonormal basis matrix of $\cX_1$. Let $\wtd U$ be an
orthonormal basis matrix of $\cY_1$. As we just argued,
$$
\wtd U^{\T}U=(\wtd U^{\T}X_1)(X_1^{\T}X_1)^{-1/2}
$$
is nonsingular, implying $\wtd U^{\T}X_1$ is nonsingular. Now take $Y_1=\wtd U(\wtd U^{\T}X_1)^{-\T}$.

Finally for item (c), let $\wtd V,\,V\in\bbR^{n\times (n-r)}$ be any orthonormal basis matrices
of $\cX_1^{\bot}$ and $\cY_1^{\bot}$, the orthogonal complements of $\cX_1$ and $\cY_1$, respectively, i.e.,
$$
\wtd V^{\T}\wtd V=V^{\T}V=I_{n-r}, \quad \wtd V^{\T}X_1=V^{\T}Y_1=0.
$$
We claim that $X:=[X_1,V]\in\bbR^{n\times n}$ is nonsingular; otherwise there exists
$$
0\ne\bx=\begin{bmatrix}
      \bz \\
      \by
    \end{bmatrix},\quad\bz\in\bbR^r,\,\,\by\in\bbR^{n-r}
$$
such that $X\bx=0$, i.e., $X_1\bz+V\by=0$, pre-multiplying which by $Y_1^{\T}$ leads to $\bz=0$, which implies $V\by=0$,
which implies $\by=0$ because $V$ is an orthonormal basis matrix of $\cY_1^{\bot}$, which says $\bx=0$,
a contradiction. Similarly, we know $Y:=[Y_1,\wtd V]\in\bbR^{n\times n}$ is nonsingular, and so is
$$
Y^{\T}X=[Y_1,\wtd V]^{\T}[X_1,V]=\begin{bmatrix}
                             Y_1^{\T}X_1 & Y_1^{\T}V \\
                             \wtd V^{\T}X_1 & \wtd V^{\T}V
                           \end{bmatrix}=\begin{bmatrix}
                                          I_r & 0 \\
                                          0 & \wtd V^{\T}V
                                        \end{bmatrix},
$$
implying $\wtd V^{\T}V$ is nonsingular. Now take
$X_2=V(\wtd V^{\T}V)^{-1}$ and $Y_2=\wtd V$.
\end{proof}

\section{Perturbation for Lyapunov equation}\label{sec:pert-Lyap}
In this section, we will establish a lemma on the change of the solution to
\begin{equation}\label{eq:Lyap-almost}
A^{\HH}X+XA+W=0
\end{equation}
subject to perturbations to $A$ and $W$, along the technical
line of \cite{heke:1988}, where $W$ may not necessarily be Hermitian. It is known as {\em the Lyapunov equation\/}
if $W$ is  Hermitian,
but here it may not be.  The result in the lemma below is used during our intermediate estimates of transfer function.
In conforming to \cite{heke:1988}, we will state the result for complex matrices: $\bbC^{n\times n}$ is the set of all $n$-by-$n$
complex matrices and $A^{\HH}$ denotes the   conjugate transpose of $A$.


\begin{lemma}\label{lem:sep}
Suppose that $A\in \bbC^{n\times n}$ is stable, i.e., all of its eigenvalues are located in the left half of the complex plane, and let
$$
K=\int_{0}^\infty e^{  A^{\HH}t}e^{At}dt,
$$
which is the unique solution to the Lyapunov equation $A^{\HH}X+XA+I_n=0$.
      Let $W\in\bbC^{n\times n}$ (not necessarily Hermitian) and $X\in \bbC^{n\times n}$ is the solution to the matrix equation
      \eqref{eq:Lyap-almost}.
Perturb $A$ and $W$ to $A+\Delta A_i$ ($i=1,2$) and $W+\Delta W$, respectively, and suppose that the perturbed equation
\begin{equation}\label{eq:pert:Lyap-almost}
(A+\Delta A_1)^{\HH}(X+\Delta X)+(X+\Delta X)(A+\Delta A_2)+(W+\Delta W)=0,
\end{equation}
has a solution $X+\Delta X$, where the trivial case either $A=0$ or $W=0$ is excluded. If
\begin{equation}\label{eq:eta:pert:Lyap}
\eta:=\|K\|_2\sum_{i=1}^2 {\|\Delta A_i\|_2}<1,
\end{equation}
then for any unitarily invariant norm $\|\cdot\|_{\UI}$
\begin{equation}\label{eq:pertLypsol}
\|\Delta X\|_{\UI}
     \le\frac {\|K\|_2}{1-\eta}\left(\|\Delta W\|_{\UI}+\|X\|_{\UI}\sum_{i=1}^2\|\Delta A_i\|_2\right).
\end{equation}
\end{lemma}

Equation \eqref{eq:Lyap-almost} is not necessarily a Lyapunov equation because $W$ is allowed non-Hermitian, not to mention
\eqref{eq:pert:Lyap-almost} for which two different perturbations are allowed to $A$ at its two occurrences.
Equation \eqref{eq:Lyap-almost} has a unique solution $X$ because $A$ is assumed stable, but a solution to the perturbed equation
\eqref{eq:pert:Lyap-almost} is assumed to exist. It is not clear if the assumption \eqref{eq:eta:pert:Lyap} ensures both $A+\Delta A_i$ for $i=1,2$ are stable
and thereby guarantees that \eqref{eq:pert:Lyap-almost} has a unique solution, too, something worthy further investigation.

\begin{proof}[Proof of \Cref{lem:sep}]
Modifying the proof of \cite[Theorem 2.1]{heke:1988}, instead of
\cite[Ineq.~(2.11)]{heke:1988} there, we have
$$
\|\Delta X\|_{\UI}\le\left(\|\Delta W\|_{\UI}+\sum_{i=1}^2\|\Delta A_i\|_2\big[\|X\|_{\UI}+\|\Delta X\|_{\UI}\big]\right)\|K\|_2,
$$
yielding \eqref{eq:pertLypsol}.
\end{proof}

In \cite{heke:1988} for the case $\Delta A_i$ for $i=1,2$ are the same and denoted by $\Delta A$, under the condition of \Cref{lem:sep} but without assuming \eqref{eq:eta:pert:Lyap}, it is proved that
\begin{equation}\label{eq:heke1988-bd}
\frac{\|\Delta X\|_2}{\|X+\Delta X\|_2}
     \le 2\|A+\Delta A\|_2 \|K\|_2\left(\frac{\|\Delta A\|_2}{\|A+\Delta A\|_2}+\frac{\|\Delta W\|_2}{\|W+\Delta W\|_2}\right),
\end{equation}
elegantly formulated in such a way that all changes are measured relatively.
We can achieve the same thing, too. In fact, under the condition of \Cref{lem:sep} but without assuming \eqref{eq:eta:pert:Lyap},
it can be shown that
\begin{equation}\label{eq:pertLypsol'}
\frac{\|\Delta X\|_{\UI}}{\|X+\Delta X\|_{\UI}}
     \le \sum_{i=1}^2\|A+\Delta A_i\|_2 \|K\|_2\left(\frac{\|\Delta A_i\|_2}{\|A+\Delta A_i\|_2}
                +\frac{\|\Delta W\|_{\UI}}{\|W+\Delta W\|_{\UI}}\right).
\end{equation}
But, as we argued at the beginning, \eqref{eq:pertLypsol} is more convenient for us to use
in our intermediate estimations.

\def\noopsort#1{}\def\l{\char32l}\def\v#1{{\accent20 #1}}
  \let\^^_=\v\def\hbk{hardback}\def\pbk{paperback}

%
%

\end{document}

%% file: defnAMS.tex
\usepackage{graphicx}
\usepackage{amsmath,amsfonts,amsthm}
\usepackage{amssymb,latexsym}
\usepackage{fixmath}
\usepackage{mathrsfs,amsbsy}
\usepackage{dsfont}
\usepackage{enumerate}
\usepackage{algorithm,algorithmic}
\usepackage{kbordermatrix}
\usepackage{xcolor}

\def\vextra{\vphantom{\vrule height0.45cm width0.9pt depth0.1cm}}

\def\wtd{\widetilde}
\def\what{\widehat}

\DeclareMathOperator{\diag}{diag}

\DeclareMathOperator{\eig}{eig}

\DeclareMathOperator{\rank}{rank}

\DeclareMathOperator{\tr}{tr}

\DeclareMathOperator{\F}{F}
\DeclareMathOperator{\HH}{H}
\DeclareMathOperator{\T}{T}
\DeclareMathOperator{\UI}{ui}

\def\bbC{\mathbb{C}}

\def\bbR{\mathbb{R}}

\def\cR{{\cal R}}

\def\cX{{\cal X}}
\def\cY{{\cal Y}}

\def\be{\pmb{e}}

\def\bu{\pmb{u}}

\def\bx{\pmb{x}}
\def\by{\pmb{y}}
\def\bz{\pmb{z}}

\def\bU{\pmb{U}}

\def\bY{\pmb{Y}}

\def\sss{\scriptscriptstyle}

\allowdisplaybreaks